\documentclass[10pt,a4paper]{amsart}

\usepackage{amsthm}
\usepackage{amscd}
\usepackage[final]{graphics}
\usepackage{amssymb,amsmath}  
\usepackage{latexsym} 
\usepackage[english]{babel}
\usepackage{rotating}
\usepackage{a4}

\usepackage{tabularx}
\usepackage{longtable}

\newtheorem{lem}{Lemma}[section]
    \newtheorem{prop}[lem]{Proposition}
    \newtheorem{thm}[lem]{Theorem}

   \theoremstyle{definition}
    \newtheorem{dfn}[lem]{Definition}

\theoremstyle{remark}
    \newtheorem{rem}[lem]{Remark}

\numberwithin{equation}{section}

\DeclareMathOperator{\Sym}{Sym}

\def\bb{\mathbf}
\def\la{\langle}
\def\ra{\rangle}
\def\phi{\varphi}

\newcommand{\Pp}[1]{\bb{P}^{#1}}

\newcommand{\GL}{\mathrm{GL}}
\newcommand{\PGL}{\mathrm{PGL}}
\newcommand{\ee}{\mathrm{e}}

\newcommand{\Ee}{\mathrm{E}}
\newcommand{\Ef}{\bb E}

\def\Z{\bb{Z}}
\def\Q{\bb{Q}}

\def\C{\bb{C}}

\newcommand{\qmin}{\mathcal{Q}^-}

\newcommand{\imin}{\mathcal{I}^-}
\newcommand{\pmin}{\pi^-}

\def\co{\colon\thinspace}
\newcommand{\num}[1]{{\operatorname{#1}}}
\newcommand{\Chi}{\mathrm{X}}
\def\pu{\bullet}
\def\mil{\mspace{-9mu}}

\newcommand{\Mm}[2]{\mathcal M_{{#1},{#2}}}

\newcommand{\s}{\mathfrak S}
\DeclareMathOperator{\Aut}{Aut}

\newcommand{\Ll}{\mathbf L}

\newcommand{\Schur}[1]{\bb S_{#1}}
\def\Inv{\Schur 2}
\def\Ant{\Schur{1,1}}
\newcommand{\schur}[1]{s_{#1}}
\def\inv{\schur 2}
\def\ant{\schur{1,1}}

\newcommand{\coh}[3][\Q]{H^{#2}({#3};{#1})}
\newcommand{\BM}[3][\Q]{\bar{H}_{#2}({#3};{#1})}

\def\ba{\big|}
\def\X{\mathcal{X}}
\DeclareMathOperator{\Conf}{Conf}
\newcommand{\B}[2]{B({#2},{#1})}
\newcommand{\tB}[2]{\tilde{B}({#2},{#1})}
\newcommand{\F}[2]{F({#2},{#1})}
\newcommand{\tF}[2]{\tilde{F}({#2},{#1})}

\newcommand{\ideal}[1]{\bb I(#1)}
\newcommand{\van}[1]{\bb V(#1)}

\def\op{\mathring}
\def\duale{\check{\ }}

\author{Orsola Tommasi}
\address{Institut f\"ur Algebraische Geometrie\\
Leibniz Universit\"at Hannover\\
Wel\-fen\-gar\-ten~1\\
D--30167 Hannover\\
Germany}
\email{tommasi@math.uni-hannover.de}
\thanks{Partial support from program \emph{Wege in die Forschung II} of Leibniz Universit\"at Hannover during the preparation of this paper is gratefully acknowledged}

\title[Quartic curves with an odd theta characteristics]{Cohomology of the moduli space of smooth plane quartic curves with an odd theta characteristic}

\date{February 17th, 2010}

\begin{document}

\begin{abstract}
We compute the rational cohomology of the moduli space of non-singular non-hyperelliptic complex projective curves of genus~3 with an odd theta characteristic.
\end{abstract}
\maketitle

\section{Introduction}
The subject of this paper is the rational cohomology of the moduli space $\qmin$ of smooth plane quartic curves with an odd theta characteristic. In other words, the elements of $\qmin$ are isomorphism classes of pairs $(C,\mathcal L)$, where $C$ is a smooth plane quartic curve and $\mathcal L$ is a line bundle on $C$ such that $\mathcal L^{\otimes2}\cong\omega_C$ and the dimension of $H^0(C,L)$ is odd. It is a classical results that for a fixed quartic curve $C$ such theta characteristics $\mathcal L$ correspond to the divisors on $C$ cut by the $28$ bitangents of $C$. Therefore, we can equivalently interpret $\qmin$ as the moduli space of pairs $(C,\tau)$ where $C$ is a smooth quartic curve and $\tau$ is a bitangent line to $C$. Note that there are two possibilities for line $\tau$ to be a bitangent line to a smooth plane quartic $C$: either $\tau$ intersects $C$ in two distinct points with multiplicity two, in which case we will call $\tau$ a \emph{proper bitangent}, of $\tau$ intersects $C$ in one point with multiplicity $4$, in which case we will call $\tau$ a \emph{flex bitangent} of $C$.

Our main result is the following.

\begin{thm}\label{main}
The cohomology with rational coefficients  of $\qmin$ is non-trivial only in degree $k\in\{0,5,6\}$. All cohomology groups carry pure Hodge structures. Specifically, one has  $\coh5\qmin=\Q(-5)$ and $\coh6\qmin=\Q(-6)^{\oplus2}$.
\end{thm}

To explain our approach, let us start by considering the moduli space $\mathcal Q$ of smooth quartic curves in the projective plane. Forgetting the chosen bitangent yields a map $o\co\qmin\rightarrow\mathcal Q$, which is finite of degree $28$. 
Quartic curves are defined by the vanishing of polynomials of degree four in three indeterminates, i.e. by elements of the vector space $V:=\C[x_0,x_1,x_2]_4$. Clearly, not every element of $V$ defines a non-singular curve, but we have to exclude the locus $\Sigma\subset V$ of singular polynomials. The action of $\GL(3)$ on $\Pp2$ and $\C[x_0,x_1,x_2]$ preserves $\Sigma$, thus inducing an action on $V\setminus\Sigma$. The moduli space $\mathcal Q$ is the geometric quotient of $V\setminus\Sigma$ by the action of $\GL(3)$.

The rational cohomology of $V\setminus\Sigma$ was computed by Vassiliev in \cite{Vart}, using his method for the computation of the cohomology of complements of discriminants. Comparing this result with the rational cohomology of the moduli space $\mathcal Q$, as computed by Looijenga in \cite{Looij}, one observes that the cohomology of the space of non-singular polynomials in $V$ is isomorphic (as graded vector space) to the tensor product of the cohomology of the moduli space  $\mathcal Q$ and that of $\GL(3)$. Indeed, Peters and Steenbrink \cite{PS} proved that this is always the case when comparing the rational cohomology of the space of non-singular homogeneous polynomials with the cohomology of the corresponding moduli space of smooth hypersurfaces. 

In this paper we use an analogous construction, in which we replace the vector space $V$ of homogeneous polynomials of degree $4$  with a certain incidence correspondence. 
This follows the approach of \cite{OTM32}, where we considered quartic curves with two marked points.
A bitangent line $\tau$ to a fixed smooth plane quartic $C$ is always uniquely determined by the scheme-theoretic intersection of $C$ and $\tau$, which is a subscheme $P\subset\Pp2$ of length $2$. Note that any $P\in\operatorname{Hilb}_2(\Pp2)$ spans a uniquely defined line $\ell_P\subset\Pp2$. If $P$ is the intersection of $C$ with a bitangent line, then this bitangent line is exactly $\ell_P$.

Therefore, we consider the incidence correspondence
$$
\imin:=
\left\{(P,f)\in\operatorname{Hilb}_2(\Pp2) \times(V\setminus\Sigma)|
f|_{\ell_P}\in \ideal P^2
\right\}.
$$
The action of $\GL(3)$ on $\Pp2$ and $V$ extends to $\imin$ and the geometric quotient $\imin/\GL(3)$ is isomorphic to $\qmin$. Then the following isomorphism of graded vector spaces with mixed Hodge structures holds:
\begin{equation*}
\coh\pu{\imin}\cong \coh\pu{\qmin}\otimes \coh\pu{\GL(3)}.
\end{equation*}
This follows from \cite{PS}, in view of \cite[Theorem~5.2]{BT}.
As a consequence, we have that determining the rational cohomology of $\imin$ immediately yields the rational cohomology of $\qmin$. 

A natural way to investigate $\imin$ and its cohomology is to use the natural projection $\pmin\co \imin\rightarrow \operatorname{Hilb}_2(\Pp2)$. First one observes that all fibres of $\pmin$ lying over reduced subschemes in $\operatorname{Hilb}_2(\Pp2)$ are isomorphic. Analogously, all fibres of $\pmin$ lying over fat points $P\in\operatorname{Hilb}_2$ are isomorphic. Furthermore, in both cases the fibres are the complement of $\Sigma$ in a linear subspace of $V$. This enable us to apply Vassiliev--Gorinov's method for the cohomology of complements of discriminants (\cite{Vart},\cite{Gorinov},\cite{OTM4}) to compute the cohomology of these fibres. The study of the Leray spectral sequence associated to the restriction of $\pmin$ to the stratum of $\imin$ corresponding to proper bitangent, respectively, to the stratum of $\imin$ corresponding to flex bitangent allows us to determine the cohomology of $\imin$.

The plan of the paper is as follows. In Section~\ref{setup} we set up our notation and we prove the relationship between the rational cohomology of the incidence correspondences we deal with and the cohomology of their $\GL(3)$-quotients. In Section~\ref{introQ0} we compute the cohomology of the moduli space of plane smooth quartic curves with a proper bitangent. The proof of this result relies on the analysis of singular configurations performed in Sections~\ref{nonrigid}--\ref{sixpts}. Finally, in Section~\ref{Qdelta} we prove that the moduli space of smooth plane quartic curves with a flex bitangent has the rational cohomology of a point. We close the paper with a brief review of Vassiliev--Gorinov's method, the topological method that plays a major role in the proof of our result.

\subsection*{Notation}
{\small
\noindent\begin{longtable}{p{1.41cm}p{10cm}}
$V$&vector space of homogeneous polynomials of degree $4$ in $x_0,x_1,x_2$.\\
$\Sigma$&locus of singular polynomials in $V$.\\
$\s_n$&the symmetric group in $n$ letters.\\
$\van f$&vanishing locus of $f$.\\
$K_0(\mathsf{HS_\Q})$& Grothendieck group of rational (mixed) Hodge structures over $\Q$.\\
$K_0(\mathsf{HS_\Q^{\s_n}})$& Grothendieck group of rational (mixed) Hodge structures endowed with an $\s_n$-action.\\
$\Q(m)$& Tate Hodge structure of weight $-2m$.\\
$\Ll$& class of $\Q(-1)$ in $K_0(\mathsf{HS_\Q})$.\\
$\Schur\lambda$& $\Q$-representation of $\s_n$ indexed by the partition $\lambda\vdash n$.\\
$\schur\lambda$& Schur polynomial indexed by the partition $\lambda\vdash n$.\\
$\Delta_j$ & $j$-dimensional closed simplex.\\
$\op\Delta_j$ & interior of the $j$-dimensional closed simplex.\\
$\F kZ$& space of ordered configurations of $k$ distinct points on the variety $Z$ (see Def.~\ref{tuples}).\\
$\B kZ$& space of unordered configurations of $k$ distinct points on the variety $Z$ (see Def.~\ref{tuples}).\\
$\pm\Q$&the twisted local system over $\B kZ$ induced by the sign representation on $\pi_1(\B kZ)$. I.e. the local system $\pm\Q$ is the rank one local system that changes its orientation under paths inducing an odd permutation of the points in the configuration.\\
$\tF 4{\Pp2}$& open subset of $\F4{\Pp2}$ such that no three points in the configuration are collinear.\\
$\tB 4{\Pp2}$& open subset of $\B4{\Pp2}$ such that no three points in the configuration are collinear.\\
$\Pp2\duale$& the dual projective plane, parametrizing all projective lines in $\Pp2$.
\end{longtable}}

Throughout this paper we will make an extensive use of Borel--Moore homology, i.e. homology with locally finite support, which we will denote by the symbol $\bar H_\pu$. A reference for its definition and the properties we use is for instance \cite[Chapter 19]{Fulton}.

To write the results on cohomology and Borel--Moore homology groups in a compact way, we will express them by means of polynomials, in the following way. Let $T_\pu$ denote a graded $\Q$--vector space with mixed Hodge structures. For every $i\in\Z$, we can consider the class $[T_i]$ in the Grothendieck group of rational Hodge structures. We define the Hodge--Grothendieck  polynomial (for short, HG polynomial) of $T_\pu$ to be the polynomial
$$\wp(T_\pu)=\sum_{i\in\Z} [T_i] t^i\in K_0(\mathsf{HS_\Q})[t].$$

If moreover the symmetric group $\s_n$ acts on $T_\pu$ respecting the grading and the mixed Hodge structures on $T_\pu$, we define the \emph{$\s_n$--equivariant HG polynomial} $\wp^{\s_n}(T_\pu)$ of $T_\pu$ by replacing $ K_0(\mathsf{HS_\Q})$ by $ K_0(\mathsf{HS_\Q^{\s_n}})$ in the definition of the HG polynomial. 

\section{Setup}\label{setup}

In this section, we establish the notation we will use in the next sections. The main ingredient of our construction is the incidence correspondence $\imin$ parametrizing pairs $(P,f)$ such that $f$ is a polynomial defining a smooth plane quartic curve and $P$ is the length two subscheme in $\Pp2$ cut on the zero locus $\van f$ by a bitangent line. We consider two natural maps on $\imin$, namely the projection $\pmin\co\imin\rightarrow\operatorname{Hilb}_2(\Pp2)$ and the quotient map $\imin\rightarrow\qmin$ by the action of $\GL(3)$. 

As explained in the introduction, we stratify $\imin$ into two strata, according to whether the bitangent line $\ell_P$ is a proper bitangent or a flex bitangent. 
For a pair $(P,f)$ in $\imin$, the bitangent line $\ell_P$ is a proper bitangent if and only if $P$ is a reduced subscheme. The locus in $\operatorname{Hilb}_2(\Pp2)$ parametrizing reduced subschemes can be identified with $\B2{\Pp2}$, the configuration space of unordered pairs of points in $\Pp2$. The complement $\operatorname{Hilb_2}(\Pp2)\setminus\B2{\Pp2}$ is the locus of fat points of multiplicity two, which is naturally isomorphic to the total space of the projectivized tangent bundle $\Pp{}(T_{\Pp2})$.

Therefore, we define the open stratum $\imin_0\subset\imin$ to be the preimage of $\B2{\Pp2}$ under $\pmin$, and the stratum $\imin_\delta$ to be the preimage of $\Pp{}(T_{\Pp2})$. 
When restricted to these two strata, the map $\pmin$ is a locally trivial fibration. We will denote the restriction of $\pmin$ to the preimages of these two strata of $\operatorname{Hilb}_2(\Pp2)$ by 
$$
\pmin_0\co \imin_0\rightarrow\B2{\Pp2},
\ \ \ \ \ \ 
\pmin_\delta\co \imin_\delta\rightarrow\Pp{}(T_{\Pp2}).
$$

Note that the quotient $\qmin_0=\imin_0/\GL(3)$ is a well-defined open subset of $\qmin$, with complement the divisor $\qmin_\delta=\qmin\setminus\qmin_0=\imin_\delta/\GL(3)$. The quotient $\qmin_0$ is the moduli space of smooth quartic curves with a marked proper bitangent, whereas $\qmin_\delta$ is the moduli space of smooth quartic curves with a marked flex bitangent, i.e. a flex line with contact of order $4$ with the curve.

In the next sections, we will compute the cohomology of $\imin_0$ and $\imin_\delta$ by using their structure as fibrations given by the maps $\pmin_0$, respectively, $\pmin_\delta$. This will allow us to obtain the cohomology of the moduli spaces $\qmin_0$ and $\qmin_\delta$ by means of the following lemma.

\begin{lem}\label{division}
The following isomorphisms of graded vector spaces with mixed Hodge structures hold:
\begin{equation}\label{icLH0}\coh\pu{\imin_0}\cong \coh\pu{\qmin_0}\otimes \coh\pu{\GL(3)}.
\end{equation}
\begin{equation}\label{icLHdelta}\coh\pu{\imin_\delta}\cong \coh\pu{\qmin_\delta}\otimes \coh\pu{\GL(3)}.
\end{equation}
\end{lem}

We recall Peters-Steenbrink's generalization of the Leray-Hirsch theorem:
\begin{thm}[\cite{PS}]\label{LeHi}
Let $\phi: X\rightarrow Y$ be a geometric quotient for the action of a connected group $G$, such that for all $x\in X$ the connected component of the stabilizer of $x$ is contractible. Consider the orbit inclusion
$$\begin{matrix}\rho_{x_0}:&G&\longrightarrow&X\\&g&\longmapsto&gx_0,\end{matrix}$$
where $x_0\in X$ is a fixed point. 
Suppose that for all $k>0$ there exist classes $e_1^{(k)},\dots,e_{n(k)}^{(k)}\in\coh kX$ that restrict to a basis for $\coh kG$ under the map induced by $\rho_{x_0}$ on cohomology. Then the map $$a\otimes\rho_{x_0}^*(e_i^{(k)})\longmapsto\phi^*a\cup e_i^{(k)}$$
extends linearly to an isomorphism of graded linear spaces
$$\coh\pu Y\otimes\coh\pu G \cong \coh\pu X$$
that respects the rational mixed Hodge structures of the cohomology groups.
\end{thm}

\begin{proof}[Proof of Lemma~\ref{division}]
Recall from \cite{PS} that the assumption of Theorem~\ref{LeHi} are satisfied for the action of $\GL(3)$ on the space $X:=V\setminus\Sigma$ of non-singular quartic polynomials. In particular, the map $\rho_{f}^*\co \coh\pu X\rightarrow \coh\pu{\GL(3)}$ associated to the orbit inclusion $\rho_f\co \GL(3)\rightarrow X$ is surjective for any choice of a base point $f\in X$. 

Next, let $(P,f)$ be a point of $\imin_0$ and  let us denote by $p_0\co \imin_0\rightarrow X$ the natural projection, which is clearly $GL(3)$-equivariant. 
 Then the orbit inclusion $\rho_f$ is the composition of $p_0$ and the orbit inclusion $\rho_{(P,f)}\co \GL(3)\rightarrow \imin_0$.  Hence, also the induced map $\rho_f^*$ in cohomology is the composition of $\rho_{(P,f)}^*$ and $p_0^*$. This implies that $\rho_{(P,f)}^*$ is surjective, so in particular it satisfies the assumptions of Theorem~\ref{LeHi}. This establishes the isomorphism~\eqref{icLH0}.

The proof of the isomorphism~\eqref{icLHdelta} is analogous, and requires to consider the orbit map $\rho_{(P,f)}\co \GL(3)\rightarrow\imin_\delta$ associated with a point $(P,f)\in\imin_\delta$.
\end{proof}

\section{Quartic curves with a proper bitangent}\label{introQ0}

In this section, we compute the rational cohomology of the moduli space $\qmin_0$ of pairs $(C,\tau)$ such that $C$ is a smooth quartic curve and $\tau$ a proper bitangent. 

\begin{thm}\label{cohq0}
The HG polynomial of the rational cohomology of $\qmin$ is equal to $1+t\Ll+t^5\Ll^5+2t^6\Ll^6$.
\end{thm}

It is known from~\cite{Harer} that the rational cohomology of $\qmin$ is trivial in degree~$1$. This implies that the cohomology class in $\coh1{\qmin_0}=\Q(-1)$ is killed by $\coh0{\qmin_\delta}=\Q$. In Theorem~\ref{cohqdelta} we will prove that this is also the only non-trivial rational cohomology group of $\qmin_\delta$. This immediately implies Theorem~\ref{main} as a corollary.

We start by considering the fibre of the map $\pmin_0\co \imin_0\rightarrow\B2{\Pp2}$ over a configuration $\{p,q\}$ of distinct points in $\Pp2$. Denote by $t$ the line $pq$ and set $t^*=t\setminus\{p,q\}.$ 
Consider the $11$-dimensional complex vector space 
$$V_{p,q}:=\left\{
f\in V\left| \begin{array}{c} 
\text{ the line $pq$ is either contained in $\van f$ or it}\\
\text{is tangent to $\van f$ at the points $p$ and $q$} \end{array}\right.\right\}.$$

Then the fibre $(\pmin_0)^{-1}(\{p,q\})$ is equal to $V_{p,q}\setminus\Sigma$. Hence, the fibre of $\pmin_0$ can be viewed as the complement of the discriminant in the vector space $V_{p,q}$. In particular, its cohomology can be computed using Vassiliev--Gorinov's method.

To apply Vassiliev--Gorinov's method to $V_{p,q}\cap\Sigma$ we need an ordered list of all possible singular sets of the elements in $V_{p,q}\cap\Sigma$. We obtain such a list by refining the list of possible singular configurations of quartic curves in \cite[Prop.~6]{Vart}. For the convenience of the reader, we copied this list in Table~\ref{Vaslist}. In the right-hand side column of that table one can read the dimension of the space of quartic polynomials which are singular at any fixed configuration of the corresponding type.

For every configuration in Table~\ref{Vaslist}, one has to distinguish further whether the singular points are or are not in general position with respect to $p$ and $q$ (for instance, if the singular configuration intersects or not the line $t:=pq$). This procedure yields a complete list of singular sets of elements of $V_{p,q}\cap\Sigma$, which we will describe in Sections~\ref{nonrigid} and~\ref{rigid}. For every type $j$ of singular configurations, we will denote by $X_j$ the space of all configurations of type $j$.

As recalled in Section~\ref{VGmethod}, Vassiliev--Gorinov's method gives a recipe to construct spaces $\ba\mathcal X\ba$ and $\ba\Lambda\ba$ and a map
$$\ba\epsilon\ba\co\ba\mathcal X\ba\rightarrow V_{p,q}\cap \Sigma$$
inducing an isomorphism in rational Borel--Moore homology. In the version of the method we use in this paper, the spaces $\ba\mathcal X\ba$ and $\ba\Lambda\ba$ are constructed as the geometric realizations of certain cubical spaces associated to the ordered list of singular sets.
The Borel--Moore homology of $\ba\mathcal X\ba$ (respectively, $\ba\Lambda\ba$) can be computed by considering the stratification $F_\pu$ (resp. $\Phi_\pu$), which is indexed by the types of configurations in the list. The properties of $F_{j}$ and $\Phi_j$ associated to the configuration type $j$ are explained in Proposition~\ref{ucci}. Recall in particular that $F_j$ is the total space of a vector bundle over $\Phi_j$, and that for finite configurations $\Phi_j$ is the total space of a (non-orientable) bundle in open simplices over the configuration space $X_j$. As a consequence, for finite configurations the Borel--Moore homology of $\Phi_j$ coincides (after a shift in the indices) with the Borel--Moore homology of $X_j$ with coefficients in a rank $1$ local system changing its orientation every time two points in a configuration are interchanged.  We will call this local system the \emph{twisted local system} $\pm\Q$.

\begin{table}
\caption{\label{Vaslist}Singular sets in $\Pp2$ of quartic homogeneous polynomials according to \cite[Prop.~6]{Vart}.}

\begin{tabular}{rp{10cm}|l}
1&Any points in $\Pp2$&$\C^{12}$\\
2&Any pair of points in $\Pp2$&$\C^9$\\
3&Any three points on the same line in $\Pp2$&$\C^7$\\
4&Any triple of non-collinear points in $\Pp2$&$\C^6$\\
5&Any line in $\Pp2$& $\C^6$\\
6&Any three points on the same line $\ell$ plus a point outside $\ell$&$\C^4$\\
7&Any quadruple of points, no three of them collinear&$\C^3$\\
8&The union of a line in $\Pp2$ and a point outside it&$\C^3$\\
9&Five points $\{a,b,c,d,e\}$ such that $\{e\} = {ab}\cap{cd}$.&$\C^2$\\
10&Six points which are the pairwise intersection of four lines in general position&$\C$\\
11&Any non-singular quadric in $\Pp2$&$\C$\\
12&The union of two lines in $\Pp2$&$\C$\\
13&The entire projective plane&$0$\\
\end{tabular}
\end{table}

It is also possible to compute directly the cohomology of $\imin_0$, without having to pass through the study of the fibre of $\pmin_0$. 
Namely, consider the space
$$\mathcal D^-_0:=\left\{
\begin{array}{c}(\{\alpha,\beta\},f)\in\F2{\Pp2}\times\Sigma: 
\text{ the line $\alpha\beta$ is either contained}\\
\text{in $\van f$ or it is tangent to it at the points $\alpha$ and $\beta$} \end{array}\right\}.$$

Note that $\mathcal D^-_0$ is a closed subset of 
$$\mathcal V^-_0:=\left\{\begin{array}{c}(\{\alpha,\beta\},f)\in\B2{\Pp2}:\text{ the line $\alpha\beta$ is either contained}\\
\text{in $\van f$ or it is tangent to it at the points $\alpha$ and $\beta$}\end{array}\right\}.$$
 The space $\mathcal V^-_0$ is the total space of a vector bundle over $\Pp2$, and $\imin_0=\mathcal V^-_0\setminus\mathcal D^-_0$. 
Vassiliev-Gorinov's method can be exploited to compute the Borel--Moore homology of $\mathcal D^-_0$. This is done by defining the singular locus of an element $(\{\alpha,\beta\},f)$ in $\mathcal D^-_0$ as the subset $\left\{\{\alpha,\beta\}\right\}\times K_f$ of $\B2{\Pp2}\times \Pp2$, where $K_f$ denotes the singular locus of the polynomial $f$. In particular, the classification of singular sets of elements of $\mathcal D^-_0$ is obtained by the classification of singular sets of elements of $V_{p,q}\setminus\Sigma$ by allowing the pair $\{p,q\}$ to move in $\B2{\Pp2}$.

Even though this is no longer the original setting of Vassiliev-Gorinov's method, one can mimic the construction of the cubical spaces $\Lambda$ and $\mathcal X$ (see Section~\ref{VGmethod}) and obtain cubical spaces $\Lambda'$ and $\mathcal X'$ that play an analogous role. In particular, the map $\ba\mathcal X'\ba\rightarrow \mathcal D^-_0$ induces an isomorphism on the Borel--Moore homology of these spaces, because it is a proper map with contractible fibres.
Moreover, for the stratifications $\Phi'_\pu$ and $F'_\pu$ obtained from the construction of $\Lambda'$ and $\mathcal X'$ we have natural maps $\Phi'_k\rightarrow\B2{\Pp2}$ and $F'_k\rightarrow\B2{\Pp2}$ which are locally trivial fibrations with fibre isomorphic to $\Phi_k$, respectively, $F_k$. 

In the next sections, we proceed by giving the classification of the singular sets in $\Pp2$ of quartic curves that are tangent to $t$ at $p$ and $q$. These are exactly the singular sets of the elements of $V_{p,q}\cap\Sigma$. In view of the discussion above, this classification also yields the classification of the singular sets in $\B2{\Pp2}\times\Pp2$ of the elements of $\mathcal D^-_0$.

Before giving the refined list, we briefly comment about which types of singular configurations will arise. A first distinction is between configurations containing a finite number of points versus configurations containing curves. In the specific case of plane quartics, singular curves are always rational. In particular, one can apply Lemma~\cite[2.17]{OTM4} (and the remarks following it) to conclude that all strata $\Phi_j$ and $F_j$ have trivial Borel--Moore homology for $j$ a type of configuration which contain rational curves. Hence, it is important to concentrate on finite configurations.

A further distinction is whether the stabilizer of a general configuration of type~$j$ in $\PGL(3)$ is finite or not. We will call configuration types with finite (resp., infinite) stabilizer \emph{rigid configurations} (resp., \emph{non-rigid} configurations). 
Typically, non-rigid finite configurations will contain few singular points which will be relatively free to move. Anyway, it is important to notice that non-rigid configurations will give a non-trivial contribution only if they contain very few points. This follows from Lemma~\ref{lem1}, which ensures that the twisted Borel--Moore homology of configurations of more than one point in affine space vanishes in all degrees, and that the same is true for $\B k{\Pp1}$ for $k\geq 3$ and $\B k{\Pp2}$ for $k\geq 4$. 
We will deal with non-rigid configurations in Section~\ref{nonrigid}.

The main result is the following:
\begin{prop}\label{nonrigid-summary}
Let us denote by $F_{\num{nrig}}\subset\ba\mathcal X\ba$ the union of the strata corresponding to non-rigid configurations (for the precise definition of these, see Sect.~\ref{list-nonrigid}). Then the $\s_2$-equivariant HG polynomial of $F_{\num{nrig}}$ with respect to the $\s_2$-action generated by the interchange of the points $p$ and $q$ is given by
$$
(3\inv+\ant)t^{20}\Ll^{-10}+
(3\inv+3\ant)t^{19}\Ll^{-9}+
(\inv+3\ant)t^{18}\Ll^{-8}+
\ant t^{20}\Ll^{-7}.
$$
\end{prop}

At the other end of the spectrum one finds rigid configurations. As we will see in Section~\ref{rigid}, if configurations of type $j$ are rigid, then the Borel--Moore homology of the strata $\Phi'_j$ and $F'_j$ is automatically a tensor product of the Borel--Moore homology of $\PGL(3)$. For this reason, for rigid configurations it is practical to work directly with the configuration space $X'_j\subset\B 2{\Pp2}\times\Pp2$ rather than with the configuration space $X_j\subset\Pp2$. As we explained above, the relationship between the two is that $X'_j$ is fibred over $\B2{\Pp2}$ with fibre isomorphic to $X_j$. 
We will investigate the contribution of rigid configurations in Section~\ref{rigid}, where we will prove the following
\begin{prop}\label{rigid-summary}
The HG polynomial of $F'_{\num{rig}}:=\ba\mathcal X'\ba\setminus F'_{\num{nrig}}$ equals
$$
t^5(1+2t\Ll)\cdot\wp(\BM\pu{\GL(3)}).
$$
\end{prop}

\begin{lem}\label{b2p2}
The Borel--Moore homology with constant (respectively, twisted) rational coefficients of the space $\B2{\Pp2}$ of unordered configurations of $2$ distinct points on $\Pp2$ is given by 
$$
\wp(\BM\pu{\B2{\Pp2}})=(1+t^2\Ll^{-1}+t^4\Ll^{-2})t^4\Ll^{-2},$$
respectively, by
$$
\wp(\BM[\pm\Q]\pu{\B2{\Pp2}})=(1+t^2\Ll^{-1}+t^4\Ll^{-2})t^2\Ll^{-1}.
$$
\end{lem}

\proof
The space $\B2{\Pp2}$ is fibred over the space $\Pp2\duale$ of lines in $\Pp2$ by the map $\{p,q\}\rightarrow pq$, which is $\s_2$-equivariant. The fibre is isomorphic to the configuration space $\B2{\Pp1}$. Then the claim follows from $\wp(\BM\pu{\B2{\Pp1}})=t^4\Ll^{-2}$ and $\wp(\BM[\pm\Q]\pu{\B2{\Pp1}})=t^2\Ll^{-1}$ (see Lemma~\ref{lem1}).
\qed

The proof of Theorem~\ref{cohq0} follows from the last two parts of the following lemma.

\begin{lem}
\begin{enumerate}
\item\label{q0-uno}
 The differentials $\delta_k$ in the long exact sequence in Borel--Moore homology
\begin{equation}\label{les-nrig}
\cdots\rightarrow
\BM {k+1}{\mathcal D^-_0}\rightarrow
\BM {k+1}{F'_{\num{rig}}}\xrightarrow{\delta_k}
\BM {k}{F'_{\num{nrig}}}\rightarrow
\BM k{\mathcal D^-_0}\rightarrow\cdots
 \end{equation}
associated with the inclusion $F'_{\num{nrig}}\subset\ba\mathcal X'\ba$ and the augmentation $\epsilon'\co\ba\mathcal X'\ba\rightarrow\mathcal D^-_0$ vanish for all indices $k$.
\item\label{q0-due}
The contribution of non-rigid configurations to the cohomology of $\imin_0$ has HG polynomial $1+t\Ll$.
\item\label{q0-tre}
The contribution of rigid configurations to the cohomology of $\imin_0$ has HG polynomial $t^5\Ll^5+2t^6\Ll^6$.
\end{enumerate}
\end{lem}

\proof
Recall from Lemma~\ref{division} that the cohomology of $\imin_0$ is a tensor product of the cohomology of $\GL(3)$. There are two equivalent ways to compute the cohomology of $\imin_0$. One possibility is to compute the Borel--Moore homology of $\mathcal D^-_0$ by using the long exact sequence~\eqref{les-nrig} and successively calculate the Borel--Moore homology of $\imin_0$ by the long exact sequence
\begin{equation}\label{incld0}
\cdots\rightarrow
\BM {k}{\mathcal D^-_0}\rightarrow
\BM {k}{\mathcal V^-_0}\rightarrow
\BM {k}{\imin_0}\rightarrow
\BM {k-1}{\mathcal D^-_0}\rightarrow
\cdots
 \end{equation}

Since $\mathcal V^-_0$ is a complex vector bundle of rank $11$ over $\B2{\Pp2}$, its Borel--Moore homology is equal to $\BM{\pu-22}{\B2{\Pp2}}\otimes\Q(11)$. In particular, from Lemma~\ref{b2p2} it follows that $\BM {k}{\imin_0}\rightarrow\BM {k-1}{\mathcal D^-_0}$ is an isomorphism for $k\leq 25$. If we compare this with the Borel--Moore homology of $F'_{\num{rig}}$ as given in Proposition~\ref{nonrigid-summary}, we have that all Borel--Moore homology classes of $\mathcal D^-_0$ coming from $\BM\pu{F'_{\num{rig}}}$ via the long exact sequence~\eqref{les-nrig} are in this range. Finally, since $\imin_0$ is smooth and $15$-dimensional , its cohomology is related to the Borel--Moore homology by 
\begin{equation}\label{BMvscoh}
\coh\pu{\imin_0} = \BM{30-\pu}{\imin_0}\otimes\Q(-15).
\end{equation}
Applying this to $\BM\pu{F'_{\num{rig}}}$, one gets that its contribution to the cohomology of $\imin_0$ is as described in~\eqref{q0-tre}, provided $\delta_k$ is trivial for all $k\leq 25$.

Another way to compute the cohomology of $\imin_0$  is to use Vassiliev--Gorinov's method to compute the Borel--Moore homology of the discriminant $V_{p,q}\cap \Sigma$, then Alexander's duality~\eqref{alexander} to deduce from this the cohomology of its complement $V_{p,q}\setminus\Delta$, and finally compute the cohomology of $\imin_0$ using the Leray spectral sequence associated to $\pmin_0\co\imin_0\rightarrow\B2{\Pp2}$. If we follow this program for the contribution of rigid configurations, the $E_2$ term of the Leray spectral sequence in cohomology associated to $\pmin_0$ is as given in the first part of Table~\ref{first-leray}. 

The information given so far determines the contribution of rigid and non-rigid configurations, up to the $d_2$ differentials of the Leray spectral sequence associated with $\pmin_0$ and the computation of the kernel of the maps $\delta_k$ of \eqref{les-nrig}. At this point, it is important to keep in mind that the rational cohomology of $\imin_0$ has to be a tensor product of the cohomology of $\GL(3)$, whose HG polynomial is $(1-t\Ll)(1-t^3\Ll^2)(1-t^5\Ll^3)$. Then one discovers that the only possibility to obtain $\coh\pu{\imin_0}$ with a structure as tensor product of $\coh\pu{\GL(3)}$ is that all maps $\delta_k$ are $0$  and that all $d_2$ differentials in the Leray spectral sequence in Table~\ref{first-leray} have the maximal possible rank. The triviality of the maps $\delta_k$ yields part \eqref{q0-uno} of the claim. The result on the rank of the differentials of the Leray spectral sequence associated to $\pmin_0$ implies that the contribution of non-rigid configurations to the $E_3$ term of this spectral sequence is as given in the second part of Table~\ref{first-leray}. In particular, this yields that the contribution of non-rigid configurations to the cohomology of $\imin_0$ is as described in part~\eqref{q0-due} of the claim. 
\qed

\begin{table}
\caption{\label{first-leray}
$E_2$ and $E_3$ terms of the Leray spectral sequence of the fibration $\mathcal I\rightarrow\B2{\Pp2}$ contributed from non-rigid configurations}
$$\begin{array}{r|cccccccc}
q&&&&&&&\\[6pt]
4&0&0&\Q(-5)&0&\Q(-6)&0&\Q(-7)\\
3&\Q(-3)  &0&\Q(-4)^4&0&\Q(-5)^4&0&\Q(-6)^3\\
2&\Q(-2)^3&0&\Q(-3)^6&0&\Q(-4)^6&0&\Q(-5)^3\\
1&\Q(-1)^3&0&\Q(-2)^4&0&\Q(-3)^4&0&\Q(-4)  \\
0&\Q&0&\Q(-1)&0&\Q(-2)&0&0
\\\hline
&0&1&2&3&4&5&6&p\\
\end{array}
$$

$$\begin{array}{r|cccccccc}
q&&&&&&&\\[6pt]
4&0&0&0     &0&0     &0&\Q(-7)\\
3&  0     &0&\Q(-4)  &0&\Q(-5)  &0&\Q(-6)^2\\
2&\Q(-2)  &0&\Q(-3)^2&0&\Q(-4)^2&0&\Q(-5)  \\
1&\Q(-1)^2&0&\Q(-2)  &0&\Q(-3)  &0&0       \\
0&\Q&0&0     &0&0     &0&0
\\\hline
&0&1&2&3&4&5&6&p\\
\end{array}
$$
\end{table}

\begin{proof}[Proof of Theorem~\ref{cohq0}]
The previous Lemma implies that the cohomology of $\imin_0$ is the direct sum of the contribution of non-rigid and of rigid configurations, and that its HG polynomials is $(1+t\Ll+t^5\Ll^5+2t^6\Ll^6)\cdot\wp(\coh\pu{\GL(3)})$. Then the claim follows from the isomorphism \eqref{icLH0} in Lemma~\ref{division}.
\end{proof}

\section{Non-rigid configurations}\label{nonrigid}

In this section we deals with the configuration types between $1$ and $6$ in Vassiliev's list (Table~\ref{Vaslist}). We need to refine these configuration types to get the classification of singular configurations of elements in $V_{p,q}\cap\Sigma$. For this first group of configuration types, one gets the cases which we list in Table~\ref{list-nonrigid}. 
In that list, we maintain the reference to the corresponding types in Vassiliev's list (Table~\ref{Vaslist}) by indicating the refined strata by roman letters. Furthermore, when it is convenient to group refined strata together, we will denote them collectively by the letter $\num x$. 

In Table~\ref{list-nonrigid} we also describe the configuration spaces $X_{j\num k}$ corresponding to each refined configuration type $j\num k$ and the associated strata $\Phi_{j\num k}\subset\ba\Lambda\ba$ and $F_{j\num k}\subset\ba\mathcal X\ba$. 
From the this description one finds that configurations of types from $1\num a$ to $6\num x$ either are non-rigid, or give strata $F_{j\num k}$ and $\Phi_{j\num k}$ which have trivial Borel--Moore homology. 

\begin{table}
\caption{\label{list-nonrigid}
Configurations of type $1$--$6$ (non-rigid configurations) and the associated strata.}
{\small
\begin{tabular}{l@{\ }p{11cm}}
1a&The point $p$ or the point $q$.\\      
&Stratum: $F_{1\num a}$ is a $\C^{10}$-bundle over $\Phi_{1\num a}=\{p,q\}$.\\
1b&Any point in $t^*$.\\
&Stratum: $F_{1\num b}$ is a $\C^9$-bundle over $\Phi_{1\num b}\cong\C^*$.\\
1c&Any point in $\Pp2\setminus t$.\\
&Stratum: $F_{1\num c}$ is a $\C^8$-bundle over $\Phi_{1\num  c}\cong\C^2$.\\
2a&The pair $\{p,q\}$.\\
&Stratum: $F_{2\num a}$ is a $\C^9$-bundle over $\Phi_{2\num a}\cong\op\Delta_1$.\\
2b&Any other pair of points on $t$.\\
&Stratum: $F_{2\num b}$ is a $\C^8$-bundle over $\Phi_{2\num b}$, which is a non-orientable $\op\Delta_1$--bundle over $\B2{t}\setminus\left\{\{p,q\}\right\}$.\\
2c&One point in $\{p,q\}$ and any point outside $t$.\\
&Stratum: $F_{2\num c}$ is a $\C^7$-bundle over $\Phi_{2\num c}$, which is a non-orientable $\op\Delta_1$-bundle over the disjoint union of two copies of $\C^2$.\\ 
2d&A point on $t^*$ and any point outside $t$.\\
&Stratum: $F_{2\num d}$ is a $\C^6$-bundle over $\Phi_{2\num d}$, which is a non-orientable $\op\Delta_1$-bundle over $\C^*\times\C^2$.\\
2e&Any pair of points in $\Pp2\setminus t$.\\
&Stratum: $F_{2\num e}$ is a $\C^5$-bundle over $\Phi_{2\num e}$, which is a non-orientable $\op\Delta_1$-bundle over $\B2{\C^2}$. Therefore, the Borel--Moore homology of $\Phi_{2\num e}$ and $F_{2\num e}$ is trivial. 
\\ 
3x&Any three points on the same line $\ell$ in $\Pp2$.\\
&There are different strata to consider ($\ell=t$; $p\in\ell\neq t$ and $p$ is one of the singular points; $q\in\ell\neq t$ and $q$ is one of the singular points; the three points do not contain $p$ or $q$ but $\ell$ do; $\ell\neq\{p,q\}=\emptyset$). In view of Lemma~\ref{lem1}, the Borel--Moore homology of the configuration space $\B3{\ell}$ is trivial. Hence, all these strata contribute trivially.\\
4a&$p$, $q$ and a further point outside $t$.\\
&Stratum: $F_{4\num a}$ is a $\C^6$-bundle over $\Phi_{4\num a}$, which is a non-orientable $\op\Delta_2$-bundle over $\C^2$.\\
4b&Any other pair of points on $t$ and a point outside it.\\
&Stratum: $F_{4\num b}$ is a $\C^5$-bundle over $\Phi_{4\num b}$, which is a non-orientable $\op\Delta_2$--bundle over $\C^2\times(\B2{t}\setminus\left\{\{p,q\}\right\})$.\\
4x&Any other triple of non-collinear points in $\Pp2$.\\
& There are several cases to consider ($p$ or $q$ and two more points lying outside $t$; one point on $t$ and two points outside; three points outside $t$).
All of them contribute trivially in view of Lemma~\ref{lem1}.\ref{lem1i}.\\
5x&Any line $\ell$ in $\Pp2$.\\
&There are several cases to be considered, according to whether the line $\ell$ is $t$, it passes through $p$ of $q$, or it is in general position with respect to $p,q$. Observe that for every line $\ell$ in $\Pp2$, all subsets of $\ell$ of cardinality at most $3$ belong to configurations of type $1$--$4$. This allows us to apply Lemma~\cite[2.17]{OTM4} and conclude that the contribution of all strata of type $5\num x$ is trivial.
\\
6x&Any three points on the same line $\ell$ plus a point outside $\ell$.\\
&Several cases, all of them do not contribute. The proof is analogous to case $3\num x$.\\
\end{tabular}
}
\end{table}

Next, we compute the contribution of non-rigid configurations to the spectral sequence $\ee^r$ converging to $\BM\pu{\ba\Lambda\ba\setminus\Phi_{13}}$ associated with the stratification $\Phi_\pu$ indexed by the configuration types. Hence, the $\ee^1$ term of this spectral sequence is given by $\ee^1_{u,v}=\BM{u+v}{\Phi_{u}}$, where $u$ refers to the $u$th configuration type in our list. Rigid configurations contribute the first nine non-trivial columns, and specifically, to the configuration types $1\num a$, $1\num b$, $1\num c$, $2\num a$, $2\num b$, $2\num c$, $2\num d$, $4\num a$ and $4\num b$. For the sake of simplicity, we will omit from the spectral sequence all configuration types $j\num k$ such that the Borel--Moore homology of $\Phi_{j\num k}$ is trivial. 

Then one gets from the description of the strata given in Table~\ref{list-nonrigid} that $\ee^1_{u,v}$ for $1\leq u\leq 9$ is as in the first part of Table~\ref{first}.

\begin{table}
  \caption{\label{first} $E^1$ terms of the spectral sequences $\ee^r_{u,v}\Rightarrow\BM{u+v}{\ba\Lambda\ba\setminus\Phi_{13}}$ and $\Ee^r_{u,v}\Rightarrow \BM\pu{\ba\mathcal X\ba}=\BM\pu{V_{p,q}\cap\Sigma}$.}
\begin{center}
\begin{tabular}{c}
$\ee^1_{u,v}$ for $1\leq u\leq 9$.\\
{\scalebox{0.8}{
$\begin{array}{r|cccccccccc}
v&&&&&&&\\[6pt]
 1&0&0&\Inv(2)&0&0&0&0&0&0&\\
 0&0&\Inv(1)&0&0&0&0&\Inv(3)&0&0&\\
-1&\Inv+\Ant&\Ant&0&0&0&(\Inv+\Ant)(2)&\Ant(2)&0&\Inv(3)&\\
-2&0&0&0&0&\Inv(1)&0&0&\Ant(2)&\Ant(2)&\\
-3&0&0&0&\Ant&\Ant&0&0&0&0&\\
\hline
&1&2&3&4&5&6&7&8&9&u\\
\num{type}&\num{(1a)}&\num{(1b)}&\num{(1c)}&\num{(2a)}&\num{(2b)}&\num{(2c)}&\num{(2d)}&\num{(4a)}&\num{(4b)}&
\end{array}
$}}
\\[48pt]
$\Ee^1_{u,v}$ for $1\leq u\leq9$.\\
{\scalebox{0.8}{
$\begin{array}{r|cccccccccc}
v&&&&&&&\\[6pt]
19&(\Inv+\Ant)(10)&0&0&0&0&0&0&0&0\\
18&0&\Inv(10)&0&0&0&0&0&0&0\\
17&0&\Ant(9)&\Inv(10)&0&0&0&0&0&0\\
16&0&0&0&0&0&0&0&0&0\\
15&0&0&0&\Ant(9)&0&0&0&0&0\\
14&0&0&0&0&\Inv(9)&0&0&0&0\\
13&0&0&0&0&\Ant(8)&(\Inv+\Ant)(9)&0&0&0\\
12&0&0&0&0&0&0&\Inv(9)&0&0\\
11&0&0&0&0&0&0&\Ant(8)&0&0\\
10&0&0&0&0&0&0&0&\Ant(8)&0\\
 9&0&0&0&0&0&0&0&0&\Inv(8)\\
 8&0&0&0&0&0&0&0&0&\Ant(7)\\
\hline
&1&2&3&4&5&6&7&8&9&u\\
\num{type}&\num{(1a)}&\num{(1b)}&\num{(1c)}&\num{(2a)}&\num{(2b)}&\num{(2c)}&\num{(2d)}&\num{(4a)}&\num{(4b)}&
\end{array}
$
}}
\end{tabular}
\end{center}
\end{table}

\begin{lem}\label{first-Lambda}
The spectral sequence $\ee^r_{u,v}\Rightarrow\BM{u+v}{\ba\Lambda\ba\setminus\Phi_{13}}$ associated with the stratification $\Phi_\pu$ and converging to the Borel--Moore homology of $\ba\Lambda\ba\setminus\Phi_{13}$ satisfies 
$$\ee^{\infty}_{1,-1}=\Inv, \ 
\ee^{\infty}_{u,v}=0\text{ for }(u,v)\neq(1,-1)
$$
for $-3\leq v\leq1$, $1\leq u\leq 9$ (i.e. for all terms coming from non-rigid configurations).
\end{lem}

\proof
This is based on the fact that for every $j=1,\dots,4$, the union of the spaces $\Phi_{j\num{k}}$ in $\ba\Lambda\ba$ coincides with the contribution of configurations of type $1$--$4$ to the auxiliary Vassiliev spectral sequence in the case of unmarked quartic curves treated in \cite[Thm~3]{Vart}. Then the claim follows from Vassiliev's proof that configurations of type $1$--$4$ contribute only trivially to the Borel--Moore homology of the open stratum of the spectral sequence converging to the Borel--Moore homology of the discriminant of unmarked plane curves.
\qed

\begin{rem}\label{diff-base}
In view of Lemma~\ref{ucci}, the stratum $F_{13}$ corresponding to the configuration $\Pp2$ is an open cone over $\ba\Lambda\ba\setminus\Phi_{13}$. 
Then Lemma~\ref{first-Lambda} proves that the only contribution of configurations of type $1$--$4$ to the Borel--Moore homology of $\ba\Lambda\ba\setminus\Phi_{13}$ is to $\BM0{\ba\Lambda\ba\setminus\Phi_{13}}=0$. If we decompose the open cone $F_{13}$ over $\ba\Lambda\ba\setminus\Phi_{13}$ as the union of its vertex and a $\op\Delta_1$-bundle over $\ba\Lambda\ba\setminus\Phi_{13}$, we see that the Borel--Moore homology group $\BM1{\op\Delta_1}\otimes\BM0{\ba\Lambda\ba\setminus\Phi_{13}}$ is killed by the Borel--Moore homology of the vertex. In other words, this implies that configurations of type $1$--$4$ contribute trivially to the Borel--Moore homology of the stratum $F_{13}\subset\ba\mathcal X\ba$.
\end{rem}

We compute the contribution of non-rigid configurations to the spectral sequence $\Ee^r_{u,v}\Rightarrow\BM{u+v}{\ba\mathcal X\ba}\cong\BM{u+v}{V_{p,q}\cap\Sigma}$ associated with the stratification $F_\pu$. Again, this will give the first $9$ columns of the spectral sequence. If we restrict our considerations to these first 9 columns, we obtain a spectral sequence converging to the Borel--Moore homology of the space $F_{\num{nrig}}:=\bigcup F_{j\num k}$ where the union is over all configurations $j\num k$ between $1\num a$ and $6\num x$. 

\begin{lem}
\item
The $\Ee^1$ terms  of the spectral sequence $\Ee^r_{u,v}\Rightarrow \BM{u+v}{V_{p,q}\cap\Sigma}$ associated with the stratification $F_\pu$ for $1\leq u\leq 9$ are as given in the second part of Table~\ref{first}. 
\end{lem}
\proof
In this spectral sequence, the $\Ee^1$ term is given by $\Ee^1_{u,v}=\BM{u+v}{F_{u}}$, where $u$ refers to the $u$th configuration type in our list. Since $F_u$ is a vector bundle of a certain rank $k_u$ over $\Phi_{u}$, one has $\Ee^1_{u,v}=\ee^1_{u,v-2k_u}\otimes\Q(k_u)$. 
\qed

\begin{proof}[Proof of Proposition~\ref{nonrigid-summary}]
Since $F_{\num{nrig}}$ is the union of the strata $F_{j\num k}$ with $j\leq 6$, its Borel--Moore homology can be computed by a spectral sequence whose $E^1$ term coincides with $\Ee^1_{u,v}$ if $u\leq 9$ and is $0$ if $u\geq 10$. Hence, the $E^1$ term coincides with the $\Ee^1_{u,v}$ in the second part of Table~\ref{first}. 

We observe that for $1\leq u\leq9$ the Hodge structure in $\Ee^1_{u,v}$ is pure of weight $-2(u+v-10)$. This implies that for every $u,r$ such that $1\leq u<u+r\leq 9$, the Hodge weight of $\Ee^r_{u,v}$ and $\Ee^r_{u+r,v-r+1}$ are different, hence all $d_r$ differentials vanish in this range. From this the claim follows.
\end{proof}

\section{Rigid configurations}\label{rigid}

In this section, we refine the second part of List~\ref{Vaslist} (i.e. configuration types from $7$ to $13$) to complete the list of singular configurations we need to apply Vassiliev--Gorinov's method to $V_{p,q}\cap\Sigma$ and the incidence correspondence $\mathcal D^-_0$. As we have briefly explained in Section~\ref{introQ0}, the configuration spaces associated with the refinements of configuration types $7$--$12$ give a non-trivial contribution unless they correspond to rigid configurations, i.e. finite configurations $(\{p,q\},\{s_1,\dots,s_r\})\subset \Pp2\times\Pp2$ with finite stabilizer in $\PGL(3)$. For such configurations, the computation of the Borel--Moore homology is easier for the ``fibred'' configuration space $X'_j$ than for the configuration space $X_j$ where we assume the bitangent to be fixed. 

The refinement of configuration type $7$ (four points in general position) gives the four configuration types described in Table~\ref{rigid-7}. For all of these configurations $7\num k$, $\num k\in\{\num a,\num b,\num c,\num d\}$ we can prove that the twisted Borel--Moore homology of the associated configuration space $X'_{7\num k}$ vanishes (see in particular Lemma~\ref{hom7c} and~\ref{X'_7d}). Here we abuse notation and we define the twisted local system $\pm\Q$ for a fibred configuration space $S\subset\B 2{\Pp2}\times\B k{\Pp2}$ as the restriction to $S$ of the pull-back of the twisted local system $\pm\Q$ under the projection $\B 2{\Pp2}\times\B k{\Pp2}\rightarrow\B k{\Pp2}$.

\begin{table}
\caption{Rigid configurations of type $7$\label{rigid-7}}
{\small
\begin{tabular}{l@{\ }p{11cm}}
7a&Any quadruple of points containing $p$ and $q$. No three points in the configuration are allowed to be collinear.\\
&Stratum: $F_{7\num a}$ is a $\C^3$-bundle over $\Phi_{7\num a}$, which in turn is a non-orientable $\op\Delta_3$-bundle over the configuration space $X_{7\num a}$. 
The space $X_{7\num a}$ is isomorphic to 
$$
\left\{\{a,b\}\in\B2{\Pp2\setminus t}|\; {ab}\cap\{p,q\}= \emptyset\right\}.
$$
The complement $\Pp2\setminus t$ is isomorphic to $\C^2$, hence the twisted Borel--Moore homology of $\B2{\Pp2\setminus t}$ vanishes by Lemma~\ref{lem1}. 
Analogously, also the twisted Borel--Moore homology of $\left\{\{a,b\}\in\B2{\Pp2\setminus t}|\; {ab}\cap\{p,q\}\neq \emptyset\right\}$ vanishes, since it is a $\B2\C$-bundle over $\C\sqcup\C$. Hence the twisted Borel--Moore homology of $X_{7\num a}$ is trivial.\\
7b&Any quadruple of points of which exactly two lie on $t$. No three points are allowed to be collinear and $\{p,q\}$ cannot be contained in the configuration.\\
&Stratum: $F_{7\num b}$ is a $\C^2$-bundle over $\Phi_{7\num b}$, which has trivial Borel--Moore homology (the proof is analogous to that for case $7\num a$).\\
7c&Any quadruple $\{a,b,c,d\}$ of points lying outside $t$, such that no three points in the configuration lie on the same line, and $t$ is a common bitangent to two distinct quadrics in the pencil passing through $\{a,b,c,d\}$.\\\label{7c}
&Stratum: $F'_{7\num c}$ is a $\C$-bundle over the space $\Phi'_{7\num c}$, which is a non-orientable $\op\Delta_3$-bundle over the configuration space $X'_{7\num c}\subset \B2{\Pp2}\times\B4{\Pp2}$. In Lemma~\ref{hom7c} we will prove that the twisted Borel--Moore homology of $X'_{7\num c}$ vanishes.\\
7d&Quadruples $\{a,b,c,d\}\not\supset\{p,q\}$ of points in general linear position  such that there is a conic $C\not\supset t$ passing through $p,q,a,b,c,d$. The conic $C$ is allowed to be singular.\\\label{7d}
&The stratum $F'_{7\num d}$ is a $\C$-bundle over the space $\Phi'_{7\num d}$, which is a non-orientable $\op\Delta_3$-bundle over the configuration space $X'_{7\num d}$ studied in Section~\ref{fourptsonconic}. In Lemma~\ref{X'_7d} we will prove that the twisted Borel--Moore homology of $X'_{7\num d}$ vanishes. 
\end{tabular}
}
\end{table}

Configuration type $8$ corresponds to the union of a line $\ell$ in $\Pp2$ and a point $s$ outside it. This type gives rise to several refined configuration types. For instance, one has to distinguish if $\ell$ coincides with the bitangent $t$, if it passes through one of the bitangency points $p,q$ or through none of them. Also the point $s$ may lie on $t$, coincide with either $p$ or $q$ or simply lie on $t$. For every refined substratum $8\num k$ one has that the Borel--Moore homology of the space $\Phi_{8\num k}$ vanishes, and hence the same holds for $F_{8\num k}$. This follows from Lemma~\cite[2.17]{OTM4} and following remarks). To apply that lemma, we have to check that for every $\ell\cup\{s\}\in X_{8\num k}$ the space $\B4\ell\times\{s\}$ was contained in one of the preceding configurations $1\num a$--$7\num d$. Moreover, one has to check that for a fixed $\ell\cup\{s\}\in X_{8\num k}$ the vector subspaces $L(K)=\{f\in\Sigma\cap V_{p,q}| K_f\supset K\}$ for every $K\in\B4\ell\times\{s\}$ defines a vector bundle. This ensures the vanishing of the Borel--Moore homology of $\Phi_{8\num k}$ and $F_{8\num k}$ for all $\num k$.

The refined singular configurations of type $9$ and $10$ are described in Table~\ref{rigid-9&10}. We will calculate the contribution of these configuration types in Sections~\ref{fivepts} and~\ref{sixpts}. Configurations of type $11$ (non-singular conics in $\Pp2$ passing through $p$ and $q$) and of type $12$ (the union of two lines containing $p$ and $q$) correspond to strata $\Phi_{11}$ and $\Phi_{12}$ with trivial Borel--Moore homology. In both cases the singular configurations are (possibly reducible) rational curves. The vanishing of Borel--Moore homology follows from Lemma~\cite[2.17]{OTM4} in the case of $\Phi_{11}$ and from  Lemma~\cite[2.17]{OTM4} in the case of $\Phi_{12}$.

\begin{table}
\caption{Rigid configurations of type $9$ and $10$.\label{rigid-9&10}}
{\small
\begin{tabular}{l@{\ }p{11cm}}
9a& Configurations of five points $\{a,b,c,d,e\}$ such that $a,b\in t$ and $t\cap cd=\{e\}$.\\
 &Stratum: $F_{9\num a}$ is a $\C^2$-bundle over $\Phi_{9\num a}$, which has trivial Borel--Moore homology.
This follows from the fact that the map $X_{9\num a}\rightarrow \B2{\C^2}$ mapping $\{a,b,c,d,e\}$ as above to $\{c,d\}$ is a locally trivial fibration with fibre isomorphic to $\B2{\C}$. Since $\BM[\pm\Q]\pu{\B2\C}$ vanishes, the twisted Borel--Moore homology of $X_{9\num a}$ must vanish as well.
\\
9b&Configurations of five points $\{p,q,a,b,c\}$ with $pa\cap qb=\{c\}$.\label{9b}\\
&Stratum: $F_{9\num b}$ is a $\C^2$-bundle over $\Phi_{9\num b}$, which is a $\op\Delta_4$-bundle over the space $X_{9\num b}$ studied in Section~\ref{fivepts}.\\
9c&Configurations of five points $\{a,b,c,d,e\}$ such that $ab\cap cd=\{e\}$, $p\in ab\setminus\{e\}$, $q\in cd\setminus\{e\}$  and $\{p,q\}\not\subset\{a,b,c,d\}$.\\
&Stratum: $F_{9\num c}$ is a $\C$-bundle over $\Phi_{9\num c}$, which is a non-orientable $\op\Delta_4$-bundle over the space $X_{9\num c}$. In Section~\ref{fivepts} we will prove that the twisted Borel--Moore homology of the fibres configuration space $X'_{9\num c}$ is trivial. Note that for general configurations $K\in X_{9\num c}$, the only quartic curve which is singular in $K$ with $t=pq$ as bitangent is the degenerate conic $ab\cup cd$ with multiplicity $2$.
\\
9d&Configurations of five points $\{a,b,c,d,e\}$ with $e\in\{p,q\}$, $a,b,c,d\notin t$ such that $ab\cap cd=\{e\}$ and $t$ is tangent to the conic passing through $a,b,c,d,q$ for $e=p$, and to the conic through $a,b,c,d,p$ for $e=q$.\\
&Stratum: $F_{9\num d}$ is a $\C$-bundle over $\Phi_{9\num d}$, which is a non-orientable $\op\Delta_4$-bundle over the space $X_{9\num d}$ studied in Section~\ref{fivepts}.\\
9e&Configurations of five points $\{a,b,c,d,e\}$ with $a,b\in t$, $\{a,b\}\neq\{p,q\}$ such that $ac\cap bd=\{e\}$.\\
&Stratum: $F_{9\num e}$ is a $\C$-bundle over $\Phi_{9\num e}$, which is a non-orientable $\op\Delta_4$-bundle over the space $X_{9\num e}$ studied in Section~\ref{fivepts}.\\
10a/b&Six points which are the pairwise intersection of four lines in general position.\\
&Stratum: $F'_{10}$ is a $\C$-bundle over $\Phi'_{10}$, which is a non-orientable $\op\Delta_5$-bundle over the space $X'_{10}$. The Borel--Moore homology of the stratum $F'_{10}$ will be computed in Section~\ref{sixpts}.\\
\end{tabular}
}
\end{table}

The only remaining stratum is the stratum $F'_{13}\subset\ba\mathcal X'\ba$ corresponding to the configuration $\Pp2$. As explained in Proposition~\ref{ucci}, the stratum $F'_{13}$ is a topological open cone with vertex a point over the space $\ba\Lambda'\ba\setminus\Phi'_{13}$, which is the union of all $\Phi'_{j\num x}$ with $j\leq12$.

\begin{lem}\label{hom13}
The HG polynomial of the  Borel--Moore homology of $F'_{13}$ equals
$$t^6 \wp(\BM\pu{\PGL(3)}).$$
\end{lem}

\proof
Let us denote by $\mathrm B:= \ba\Lambda'\ba\setminus\Phi'_{13}$ the base of the open cone $F'_{13}$. We intend to compute its Borel--Moore homology by using the spectral sequence associated to the stratification $\Phi_\pu$. 
We have already proved that the Borel--Moore homology of $\Phi'_{j\num{x}}$ with $5\leq j\leq 8$ or $11\leq j\leq 12$ is trivial, either in a straightforward way because these configurations contain too many points on the same rational curve, or in the Lemmas~\ref{hom7c} and  \ref{X'_7d}.
The union $\bigcup_{\num k\in\{\num a,\num b,\num c,\num d\}}\Phi'_{9\num k}$ has trivial Borel--Moore homology in view of Lemma~\ref{union9}.
Furthermore, as we explained in Remark~\ref{diff-base}, configurations of type $1$--$4$ contribute trivially to the Borel--Moore homology of $F'_{13}$. 
From this it follows that the only strata contributing to the Borel--Moore homology of the basis  of the open cone $F'_{13}$ are $\Phi'_{9\num e}$ and $\Phi'_{10}$. Therefore, there is a long exact sequence 
$$
\cdots \rightarrow
\BM {k}{\Phi'_{9\num e}} \rightarrow
 \BM k{\mathrm B} \rightarrow
\BM k{\Phi'_{10}} \xrightarrow{\delta_k}
\BM {k-1}{\Phi'_{9\num e}} \rightarrow \cdots
$$
in Borel--Moore homology. In Lemma~\ref{hom9e} and Lemma~\ref{hom10} we prove that both the Borel--Moore homology of $\Phi'_{9\num e}$ and of $\Phi'_{10}$ are tensor products of the Borel--Moore homology of $\PGL(3)$. This is a consequence of the fact that these configurations are rigid.
Since $\PGL(3)$ acts equivariantly on the whole of $\mathcal D^-_0$ and $\ba\Lambda'\ba$, the differentials $\delta_k$ have to respect this structure as tensor products of $\BM\pu{\PGL(3)}$. In particular, in our specific case this implies that all $\delta_k$ are induced from the differential $\delta_{25}$ between the non-trivial Borel--Moore homology classes of $\Phi_{10}$ and $\Phi_{9\num e}$ in top degree.
Furthermore, the claim is equivalent to showing that the differential $\delta_{25}$ is an isomorphism.

Assume by contradiction that $\delta_{25}$ were the $0$ map. Then we would have $\BM{25}{\mathrm B}=\Q(10)$ and thus $\BM{26}{F'_{13}}=\Q(10)$ for the open cone over $\mathrm B$. By briefly comparing this with the Borel--Moore homology of the strata $F'_{j\num x}$ with $j\leq 12$, we find that the contribution of $\BM{26}{F'_{13}}$ to the spectral sequence $\Ef^r_{p,q}\Rightarrow \BM\pu{\mathcal D^-_0}$ cannot be killed by any differential of that spectral sequence. In particular, this means that $\BM{26}{\mathcal D^-_0}$ is an extension of $\Q(10)$. By duality (see~\eqref{incld0} and~\eqref{BMvscoh}), this would imply that $\coh3{\imin_0}$ is an extension of $\Q(-5)$, which is clearly impossible since the Hodge weight of $\Q(-5)$ is $10>2\cdot 3$, whereas Hodge weights in cohomology can never be larger than twice the degree. 

From this it follows that $\delta_{25}$ must have rank $1$ and $\BM\pu{\mathrm B}=\BM{\pu-5}{\PGL(3)}$. Then the claim follows from the structure of $F'_{13}$ as an open cone over $\mathrm B$.  
\qed

\begin{rem}
One can also give a direct proof of the non-vanishing of $\delta_{25}$ based on geometric considerations on the configuration spaces involved.
\end{rem}

\begin{proof}[Proof of Proposition~\ref{rigid-summary}]
We consider the spectral sequence $\Ef^r_{p,q}\Rightarrow \BM{p+q}{\mathcal D^-_0}$ with $\Ef^1_{p,q}=\BM{p+q}{F'_p}$. 
We concentrate on the rigid configuration types, i.e. those of type $j\num x$ with $7\leq j\leq 13$. Their union is the space
$F'_{\num{rig}}$ of which we want to compute the Borel--Moore homology.

In view of the results in this section and in Sections~\ref{fourptsonconic}--\ref{sixpts}, the only strata $F'_{j\num x}$ with non-trivial Borel--Moore homology are those of type $9\num x$, $10$ and $13$, whose Borel--Moore homology is computed in Lemmas~\ref{hom9b}--\ref{hom9e}, \ref{hom10} and \ref{hom13}. Furthermore, the Borel--Moore homology of each of these strata is a tensor product of the Borel--Moore homology of $\PGL(3)$. Hence, the $E^1$ terms coming from such configurations are of the form $\Ef^1_{p,\pu}
=\hat \Ef^1_{p,\pu}\otimes\BM\pu{\PGL(3)}$. We give the $\hat \Ef^1$ terms in Table~\ref{second}.
Note that, by construction, differentials should respect the structure of the columns of the spectral sequence as tensor products of the Borel--Moore homology of $\PGL(3)$. This implies that the only differential that can possibly be non-trivial is $\hat d^1\co \hat \Ef^1_{1,10}=\Q(3)\rightarrow\Ef^1_{1,9}=\Q(3)$. From the definition of $\Ef^1_{p,q}$ we have $\hat\Ef^1_{1,10}\otimes \BM{16}{\PGL(3)} = \BM{27}{F'_{10}}$ and $\hat\Ef^1_{1,9}\otimes \BM{16}{\PGL(3)} = \BM{26}{F'_{9\num e}}$. This means that $d^1$ is induced by the differential $\BM{27}{F'_{10}}\rightarrow \BM{26}{F'_{9\num e}}$ of the long exact sequence associated to the inclusion of $F'_{9\num e}$ as a closed subset of $F'_{9\num e}\cup F'_{10}$.

We claim that $d^1$ is an isomorphism or, equivalently, that the Borel--Moore homology of $F'_{9\num e}\cup F'_{10}$ vanishes in degree $27$. Indeed, the union $F'_{9\num e}\cup F'_{10}$ is a $\C$-bundle over $\Phi'_{9\num e}\cup \Phi'_{10}$, whose Borel--Moore homology in degree $25$ vanishes by the proof of Lemma~\ref{hom13}. 
From this the claim follows. In particular, the spectral sequence $\Ef^r_{p,q}$ restricted to rigid configurations types degenerates at $E^2$.
\end{proof}

\begin{table}
\caption{\label{second}$\hat \Ef^1$ terms associated with the spectral sequence converging to the Borel--Moore homology of $\mathcal D^-_0$ coming from configurations of type $5$--$13$.}

$$\begin{array}{r|cccc}
q&&&&\\[6pt]
9&\Q(3)&\Q(3)&0&\\
8&0&0&0&\\
7&\Q(2)&0&0&\\
6&\Q(1)^{\oplus2}{\mspace{-21mu}}&0&0&\\
5&0&\Q(1)&0&\\
4&0&0&0&\\
3&0&0&\Q^{\oplus2}\mil&\\
\hline
&1&2&3&p\\
\num{type}&\num{(9x)}&\num{(10)}&\num{(13)}\\
\end{array}
$$
\end{table}

\section{Configuration type $7\num c$ --- Pencils of conics}\label{fourpts}
The aim of this section is to compute the contribution of singular configurations of type $7\num c$ (see page~\pageref{7c}) to the Vassiliev spectral sequence converging to the Borel--Moore homology of the incidence correspondence $\mathcal D^-_0$.
In other words, we will compute the rational Borel--Moore homology of the spaces $\Phi'_{7\num c}$ and $F'_{7\num c}$.

\begin{lem}\label{hom7c}
The stratum $\Phi'_{7\num c}\subset\ba\Lambda'\ba$ and of the stratum $F'_{7\num c}$ of $\ba\mathcal X'\ba$ have trivial Borel--Moore homology. 
\end{lem}

\proof
We start by determining the twisted Borel--Moore homology of the underlying family of configurations $X'_{7\num c}$. Denote by $\tB4{\Pp2}$ the space of quadruples of points in general position, i.e. such that no three of the points lie on the same line. For every element $K$ of $\tB4{\Pp2}$ there is exactly one pencil of conics $M_K\subset V$ with base locus $K$. For every point $p\in\Pp2\setminus K$ we denote by $Q_{K,p}$ the unique conic in $K$ passing through $p$.

The family of configurations $X_{7\num c}\subset\B2{\Pp2}\times\tB4{\Pp2}$ is the locus of configurations $(\alpha,\beta,K)$ such that $K\cap \alpha\beta = \emptyset$ and the pencil $M_K$ contains a conic tangent to $\tau=\alpha\beta$ in $\alpha$ and a conic tangent to $\tau$ in $\beta$. 
We can rephrase this by saying that the space $X'_{7\num c}\subset\B2{\Pp2}\times\tB4{\Pp2}$ is the space of configurations $(\{\alpha,\beta\},K)$ such that the tangent line at $\alpha$ to the conic $Q_{K,\alpha}$ and the tangent line at $\beta$ to the conic $Q_{K,\beta}$ are both equal to $\tau:=\alpha\beta$. Note that two general conics admit exactly four common tangents of this type. If we fix the base locus $K$ of the pencil of conics and a line $\tau$ disjoint from $K$, we have that exactly two conics in $M_K$ are tangent to $\tau$. These two conics coincide if and only if $\tau$ intersects the base locus of the pencil $M_K$.

From the discussion above, it follows that $X'_{7\num c}$ is isomorphic to the locus $Y_{7\num c}\subset\Pp2\duale\times\tB4{\Pp2}$ of configurations $(\tau,K)$ where the line $\tau$ does not contain any of the points in $K$. The natural map $Y_{7\num c}\xrightarrow{\sim} X'_{7\num c}\rightarrow \B2{\Pp2}$ is defined by associating to $(\tau,K)$ the set of two points at which $\tau$ is tangent to a conic in $M_K$. 

Let us start from the case in which the configuration $K=\{a_1,a_2,a_3,a_4\}$ is fixed. Then the space of lines $\tau$ such that $\tau\cap K=\emptyset$ is the complement in $\Pp2\duale$ of the union of four lines $a_i\duale$ in general position, given by that pencil of lines passing through each of the $a_i$. To proceed we need to know the Borel--Moore homology of $\mathcal U:=\Pp2\duale\setminus\bigcup a_i\duale$ and its structure as representation of the symmetric group $\s_4$ given by the natural action of $\s_4$ permuting the points in $K$. This is computed in Lemma~\ref{fourlines} below. In particular, the Borel--Moore homology of $\mathcal U$ does not contain any alternating classes. Since $\mathcal U$ is isomorphic to the fibre of $X'_{7\num c}\rightarrow\tB4{\Pp2}$ and the whole Borel--Moore homology of $\tF4{\Pp2}\cong\PGL(3,\C)$ is $\s_4$-invariant, this implies that also the Borel--Moore homology of $X'_{7\num c}$ does not contain any non-trivial $\s_4$-alternating class. 

In view of the structure of $\Phi'_{7\num c}$ as non-orientable simplicial bundle over $X'_{7\num c}$, and of $F'_{7\num c}$ as vector bundle over $\Phi'_{7\num c}$, the vanishing of the twisted Borel--Moore homology of $X'_{7\num c}$ implies the vanishing of the Borel--Moore homology of  $\Phi'_{7\num c}$ and $F'_{7\num c}$ is trivial as well.
\qed

\begin{lem}\label{fourlines}
The $\s_4$-equivariant HG polynomial of $\BM\pu{\Pp2\duale\setminus\bigcup a_i\duale}$ is equal to $t^4\Ll^2\schur4+t^3\schur{3,1}\Ll+t^2\schur{3,1}$. 
\end{lem}

\begin{proof}
We start by computing the Borel--Moore homology of $\mathcal C:=\bigcup a_i\duale$. First we consider the singular locus of $\mathcal C$, which is the union of six points on which $\s_4$ acts as the representation $\Schur4\oplus\Schur{2,2}\oplus\Schur{2,1,1}$. 
For each $i$, the locus $a_i\duale\setminus \mathcal C_{\num{sing}}$ is isomorphic to $\Pp1$ minus three points. To determine the $\s_4$-action on the Borel--Moore homology of $\mathcal C \setminus \mathcal C_\num{sing}$, we start by observing that the group $\s_3$ permuting the three singular points on $a_4\duale$ acts as $\Schur 3$ on the Borel--Moore homology in degree $2$ and as $\Schur{2,1}$ in degree $1$. By extending these representations to representations of $\s_4$ we get that the Borel--Moore homology of $\mathcal C\setminus\mathcal C_\num{sing}$ is $(\Schur4\oplus\Schur{3,1})(1)$ in degree 2 and $\Schur{3,1}\oplus\Schur{2,2}\oplus\Schur{2,1,1}$ in degree $1$. In all other degrees the Borel--Moore homology is trivial.

The closed inclusion $\mathcal C_\num{sing}\rightarrow\mathcal C$ induces a long exact sequence in Borel--Moore homology which yields that the $\s_4$-equivariant HG-polynomial of $\mathcal C$ is $\schur4+\schur{3,1}t+(\schur4+\schur{3,1})t^2\Ll$. Here we used the fact that $\BM0{\mathcal C}$ is $1$-dimensional to compute the rank of the only non-trivial differential of the long exact sequence, i.e. $\BM1{\mathcal C\setminus\mathcal C_\num{sing}}\rightarrow\BM0{\mathcal C_\num{sing}}$. Then the claim follows from the long exact sequence associated to the closed inclusion $\mathcal C\rightarrow\Pp2\duale$, with complement $\mathcal U$.
\end{proof}

\section{Configuration type $7\num d$ --- Conics through $6$ points}

\label{fourptsonconic}
In this section we deal with the configuration space $X'_{7\num d}$ of semi-ordered configurations $(\{\alpha,\beta\},P=\{p_1,p_2,p_3,p_4\})$ of points in $\Pp2$, satisfying
\begin{itemize}
\item the points $p_1,p_2,p_3,p_4$ are in general position, i.e. $P\in\tB4{\Pp2}$;
\item there is a conic $C\in\mathcal K:=\Pp{}(\C[x_0,x_1,x_2]_2)$ containing $\{\alpha,\beta\}\cup P$;
\item $\{\alpha,\beta\}\not\subset P$, $\alpha\beta\not\subset C$.
\end{itemize}
In the following, we will always use the notation $\mathcal K$ for the projective space of conic curves in $\Pp2$.

We will prove the following result:
\begin{lem}\label{X'_7d}
Consider the rank $1$ local system of coefficients induces by the sign representation of the symmetric group $\s_4$ on the points in the configuration $P\in\tB4{\Pp2}$. Then the Borel--Moore homology of $X'_{7\num d}$ with $S$-coefficients vanishes. 
\end{lem}

We start with the observation that the conic $C$ on which the points $\alpha,\beta,p_1,\dots,p_4$ lie is unique for all $(\alpha,\beta,P)\in X_{7\num d}$. Therefore, one can view $X_{7\num d}$ as a subset of $\B2{\Pp2}\times\tB4{\Pp2}\times\mathcal K$. A partial compactification is given by considering the space $Y'_{7\num d}$ of configurations $$(\{\alpha,\beta\},P,C)\in\B2{\Pp2}\times\tB4{\Pp2}\times\mathcal K$$ such that $\{\alpha,\beta\}\cup P$ lie on the conic $C$. Note that the local system $S$ extends to $Y'_{7\num d}$.

Then Lemma~\ref{X'_7d} follows from the following lemma:
\begin{lem}\label{Y'_7d}
The Borel--Moore homology with $S$-coefficients of both $Y'_{7\num d}$ and $Y'_{7\num d}\setminus X'_{7\num d}$ vanishes.
\end{lem}

\proof
We can stratify $Y'_{7\num d}$ as the disjoint union of the subset $Y'_{7\num d,1}$ where the conic $C$ is singular and the locus $Y'_{7\num d,2}$  of $Y'_{7\num d}$ such that the conic $C$ has maximal rank. 
 Then the proof of the first part of the claim consists of showing that the Borel--Moore homology of each of the strata $Y'_{7\num d,j}$ vanishes. For instance, the open stratum $Y'_{7\num d,2}$ is fibred over the space of all non-singular conics in $\Pp2$ with two distinct marked points $\alpha,\beta$, and the fibre is isomorphic to $\B4C$. Then the claim follows from the fact that $C\cong\Pp1$ and $S$ restricts to the local system $\pm\Q$ on the fibre. Recall from Lemma~\ref{lem1} that the twisted Borel--Moore homology $\BM[\pm\Q]\pu{\B k{\Pp1}}$ vanishes
for $k\geq 3$. The proof of the vanishing for $Y_{7\num d,1}$ is similar, and is based on the fact that $Y_{7\num d,1}$ can be realised as a fibration with fibres isomorphic to $\B2{\C}$. From this the vanishing of the twisted Borel--Moore homology of $Y'_{7\num d}$ follows.

Next, let us consider the complement $Y'_{7\num d}\setminus X'_{7\num d}$. Since for these configurations $(\{\alpha,\beta\},P,C)$ one has $\alpha,\beta\in P$, one can view $Y'_{7\num d}\setminus X'_{7\num d}$ as the set of partially ordered configurations $(\{\alpha,\beta\},\{q_1,q_2\},C)$ in $(\tF4{\Pp2}/\sim)\times \mathcal K$ such that $C$ contains $\{\alpha,\beta,q_1,q_2\}$. The relation $\sim$ on $\tF4{\Pp2}$ is generated by $(\alpha,\beta,q_1,q_2)\sim(\beta,\alpha,q_1,q_2)$ and $(\alpha,\beta,q_1,q_2)\sim(\alpha,\beta,q_2,q_1)$ and the local system $S$ is the local system induced by the sign representation of the $\s_2$-action interchanging $q_1$ and $q_2$.

Hence, $Y'_{7\num d}\setminus X'_{7\num d}$ is a finite quotient of the subset $Z$ of $\tF4{\Pp2}\times\mathcal K$ of configurations $(\alpha,\beta,q_1,q_2,C)$ such that $C$ contains the points $\alpha,\beta,q_1,q_2$ but is different from the rank $1$ conic $C_0:=\alpha\beta\cup q_1q_2$. This space $Z$ is a rank $1$ affine bundle over $\tF4{\Pp2}$, the fibre over a configuration of four points being the pencil of conics passing through the them, minus $C_0$. From the isomorphism $\tF4{\Pp2}\cong\PGL(3)$ one gets that the whole Borel--Moore homology of $\tF4{\Pp2}$ is invariant under the interchange of two points in the configuration. Moreover, also the whole Borel--Moore homology of the fibres of the $\C$-bundle is invariant under such an interchange. Hence the Borel--Moore homology of $Y'_{7\num d}\setminus X'_{7\num d}$ with constant coefficients, which equals the part of the Borel--Moore homology of $Z$ which is invariant under the interchange of the third and four point in the configuration, is equal to the Borel--Moore homology of $Z$.
From the construction of $Y'_{7\num d}\setminus X'_{7\num d}$ as the quotient of $Z$ under an involution, and from the definition of $S$, we get  $\BM\pu Z = \BM\pu{Y'_{7\num d}\setminus X'_{7\num d}}\oplus\BM[S]\pu{Y'_{7\num d}\setminus X'_{7\num d}}$. From this the vanishing of $\BM[S]\pu{Y'_{7\num d}\setminus X'_{7\num d}}$ follows.
\qed

\section{Type 9}\label{fivepts}

In this section we compute the Borel--Moore homology of the strata $F'_{9\num k}\subset\ba\mathcal X'\ba$ with $\num k\in\{\num b,\num c,\num d,\num e\}$ that correspond to singular configurations containing the union of four points $\{a,b,c,d\}$ in general position  with the point $e$ which is the intersection of the lines $ab$ and $cd$. The configurations spaces $X_{9\num k}$ were described in Table~\ref{rigid-9&10} on page~\pageref{rigid-9&10}.

We will prove the following results:
\begin{lem}\label{hom9b}
For configurations of type $9\num b$ one has:
\begin{equation*}
\begin{array}{l}
\BM[\pm\Q]\pu {X'_{9\num b}}\cong\BM\pu{\PGL(3)},\\
\BM \pu{\Phi'_{9\num b}}\cong \BM{\pu-4}{\PGL(3)},\\
\BM \pu{F'_{9\num b}}\cong\BM{\pu-8}{\PGL(3)}\otimes \Q(2).
\end{array}
\end{equation*}
\end{lem}

\begin{lem}\label{hom9c}
For configurations of type $9\num c$ one has:
\begin{equation*}
\begin{array}{l}
\BM[\pm\Q]\pu {X'_{9\num c}}\cong\BM{\pu-1}{\PGL(3)},\\
\BM \pu{\Phi'_{9\num c}}\cong \BM{\pu-5}{\PGL(3)},\\
\BM \pu{F'_{9\num c}}\cong\BM{\pu-7}{\PGL(3)}\otimes \Q(1).
\end{array}
\end{equation*}
\end{lem}

\begin{lem}\label{hom9d}
For configurations of type $9\num d$ one has:
\begin{equation*}
\BM[\pm\Q]\pu {X'_{9\num d}}=
\BM \pu{\Phi'_{9\num d}}=
\BM \pu{F'_{9\num d}}=0.
\end{equation*}
\end{lem}
\begin{lem}\label{hom9e}
For configurations of type $9\num e$ one has:
\begin{equation*}
\begin{array}{l}
\wp(\BM[\pm\Q]\pu {X'_{9\num e}})=(t^4\Ll^{-2}+t)\cdot\wp(\BM{\pu}{\PGL(3)})\\
\wp(\BM \pu{\Phi'_{9\num e}})=(t^8\Ll^{-2}+t^5)\cdot\wp(\BM{\pu}{\PGL(3)})\\
\wp(\BM \pu{F'_{9\num e}})=(t^{10}\Ll^{-3}+t^7\Ll^{-1})\cdot\wp(\BM{\pu}{\PGL(3)})\\
\end{array}
\end{equation*}
\end{lem}

\begin{lem}\label{union9}
\begin{enumerate}
\item\label{union9-1}
The Borel--Moore homology of the union of the configurations spaces $X'_{9\num b}$ and $X'_{9\num c}$ inside $\B2{\Pp2}\times\B6{\Pp2}$ has trivial twisted Borel--Moore homology. 
\item\label{union9-2}
The Borel--Moore homology of the union of the strata $\Phi'_{9\num b}$ and $\Phi'_{9\num c}$ inside $\ba\Lambda'\ba$ is trivial.
\end{enumerate}
\end{lem}

\begin{proof}[Proof of~Lemma~\ref{hom9b}]
Recall that the configuration space $X'_{9\num b}$ is a finite quotient of the space $\tF4{\Pp2}$ of ordered configurations of points in general position, via the map 
$$\begin{array}{ccc}
\tF4{\Pp2}&\longrightarrow&X'_{9\num b}\\
(a_1,a_2,a_3,a_4)&\longmapsto&(\{a_1,a_2\},\{a_1,a_2,a_3,a_4,a_5\}),
\end{array}
$$
where the point $a_5$ is the intersection point of the lines $a_1a_3$ and $a_2a_4$.
This map can be identified with the quotient map of $\tF4{\Pp2}$ by the involution $(1,2)(3,4)$. Note that the involution $(1,2)(3,4)$ has even sign, so that the restriction to $X'_{9\num b}$ of the local system $\pm\Q$ (defined by the sign representation of the action of $\s_5$ on the $5$ singular points) equals the constant local system $\Q$. 
Then the claim follows from the fact that $\tF4{\Pp2}$ is isomorphic to $\PGL(3)$ and that the whole of its Borel--Moore homology is invariant under permutation of the points.

The result over the Borel--Moore homology of $\Phi'_{9\num b}$ follows from the fact that $\Phi'_{9\num b}$ is a $\op\Delta_4$-bundle over $X'_{9\num b}$. Note that the involution $(1,2)(3,4)$ does not change the orientation of the simplex with vertices $a_1,a_2,a_3,a_4,a_5$, so that the simplicial bundle $\Phi'_{9\num b}\rightarrow X'_{9\num b}$ is orientable in this case.
The result over the Borel--Moore homology of $F'_{9\num b}$ follows from the fact that $F'_{9\num b}\rightarrow \Phi'_{9\num b}$ is a complex vector bundle of rank $1$.
\end{proof}

\begin{proof}[Proof of Lemma~\ref{union9}]
The union $U$ of the configuration spaces $X'_{9\num b}$ and $X'_{9\num c}$ inside $\B2{\Pp2}\times\B5{\Pp2}$ is the locus of configurations $(\{\alpha,\beta\},\{a_i\}_{1\leq i\leq 5})$ such that $a_1,a_2,a_3,a_4$ are in general linear position, the point $\alpha$ belongs to $a_1a_2$, the point $\beta$ belongs to $a_3a_4$ and furthermore $a_1a_2\cap a_3a_4 = \{a_5\}\not\subset \alpha\beta$.

Consider the configuration space
$$
Y:=\left\{(\alpha,\beta,a_1,a_2,a_3,a_4)\in \F2{\Pp2}\times\tF4{\Pp2}
\left| 
\begin{array}{c}
\alpha\in a_1a_2, \beta\in a_3a_4,\\
 \alpha,\beta\not\in(a_1a_2\cap a_3a_4)
\end{array}
\right.
\right\}
$$
of ordered configurations of six points, such that the last four points $a_1,a_2,a_3,a_4$ are in general position, the first point $\alpha$ lies on $a_1a_2\setminus a_3a_4$ and the second point $\beta$ lies on $a_3a_4\setminus a_1a_2$. 
Notice that interchanging the configurations $(\alpha,\beta,a_1,a_2,a_3,a_4)$ and $(\alpha,\beta,a_2,a_1,a_3,a_4)$ gives a well defined involution on $Y$. It is easy to prove that the whole Borel--Moore homology of $Y$ is invariant with respect to this involution. Namely, the space $Y$ is fibred over $\tF4{\Pp2}\cong\PGL(3)$, whose Borel--Moore homology is invariant under the involution interchanging $a_1$ and $a_2$. 
The involution interchanging $a_1$ and $a_2$ induces a trivial action also on the Borel--Moore homology of the fibre of $Y\rightarrow\tF4{\Pp2}$, which is isomorphic to $\C^2\cong (a_1a_2\setminus\{\text{pt}\})\times (a_3a_4\setminus\{\text{pt}\})$.

The space  $U$ is the quotient of $Y$ by the group generated by the involutions
$$\begin{array}{c}
(\alpha,\beta,a_1,a_2,a_3,a_4)\leftrightarrow(\alpha,\beta,a_2,a_1,a_3,a_4),
\\ 
(\alpha,\beta,a_1,a_2,a_3,a_4)\leftrightarrow(\alpha,\beta,a_1,a_2,a_4,a_3),
\\ 
(\alpha,\beta,a_1,a_2,a_3,a_4)\leftrightarrow(\beta,\alpha,a_3,a_4,a_1,a_2).
\end{array}
$$

Thus, the twisted Borel--Moore homology of $U$ is contained in the part of the Borel--Moore homology of $Y$ which is alternating under the involution interchanging $a_1$ and $a_2$ in the configurations. This proves that the twisted Borel--Moore homology of $U$ is trivial.

The second part of the claim follows from the fact that $\Phi'_{9\num b}\cup\Phi'_{9\num c}$ is a $\op\Delta_4$-bundle over $U$.
\end{proof}

\begin{proof}[Proof of~Lemma~\ref{hom9c}]
We start by observing that the result on the twisted Borel--Moore homology of $X'_{9\num c}$ implies the result for $\BM\pu{\Phi'_{9\num c}}$ and $\BM\pu{F'_{9\num c}}$. This follows from the structure of $\Phi'_{9\num c}$ as $\op\Delta_4$-bundle over $X'_{9\num c}$ and from the structure of $F'_{9\num c}$ as rank $1$ complex vector bundle over $\Phi'_{9\num c}$. 

The configurations space $X'_{9\num c}$ is an open subset of the space $U$ of Lemma~\ref{union9}, with complement the configuration space $X'_{9\num b}$. The twisted Borel--Moore homology of $U$ is trivial, whereas the twisted Borel--Moore homology of $X'_{9\num b}$ is isomorphic to the Borel--Moore homology of $\PGL(3)$ by Lemma~\ref{hom9b}.
Then the claim follows from the long exact sequence in Borel--Moore homology with $\pm\Q$-coefficients associated to the closed inclusion $X'_{9\num b}\hookrightarrow U$.
\end{proof}

\begin{proof}[Proof of~Lemma~\ref{hom9d}]
Consider the space
$$Y_{9\num d}:=\{(a_1,a_2,a_3,a_4,\tau)\in\tF4{\Pp2}\times{\Pp2}\duale| \tau\cap a_1a_2 = \tau\cap a_3a_4= a_1a_2\cap a_3a_4\}$$
of ordered configurations of four points $a_1,a_2,a_3,a_4$ in general position, together with a line $\tau$ passing through the common point of the lines $a_1a_2$ and $a_3a_4$, and different from these two lines. The natural map $Y_{9\num d}\rightarrow\tF4{\Pp2}$ gives $Y_{9\num d}$ the structure of a $\C^*$-bundle over $\tF4{\Pp2}$. Note that the whole Borel--Moore homology of $Y_{9\num d}$ is invariant under the interchange of the points $a_1, a_2$. This involution fixes the fibres of the $\C^*$-bundle (considered as a subset of $\Pp2\duale$), and it also acts trivially on the Borel--Moore homology of the basis $\tF4{\Pp2}\cong\PGL(3)$. 

Next, fix a configuration $\underline y=(a_1,a_2,a_3,a_4,\tau)\in Y_{9\num d}$ and consider the pencil of quadrics through the points $a_1,a_2,a_3,a_4$. Every quadric in the pencil intersects the line $\tau$ in a subscheme of length $2$, and exactly two quadrics in the pencil are tangent to $\tau$, namely, the reduced conic $a_1a_2\cup a_3a_4$ and a further conic, which we will denote by $Q_{\underline{y}}$. Note that the tangency points of the two conics are distinct points on $\tau$, otherwise we would get a contradiction with the assumption that the $a_i$ are in general position.

Consider next the map 

$$\begin{array}{ccc}Y_{9\num d}&\longrightarrow&X'_{9\num d}\\
\underline{y}=(a_1,a_2,a_3,a_4,\tau)&\longmapsto&(\{\alpha,\beta\},\{a_1,a_2,a_3,a_4,\alpha\})
\end{array}$$ 
where $\{\alpha\}=a_1a_2\cap a_3a_4$ and $\beta$ is the intersection point of $Q_{\underline{y}}$ and $\tau$. This map is surjective with finite fibres, and allows to identify $X'_{9\num b}$ with the quotient of $Y_{9\num d}$ by the subgroup of $\s_4$ generated by $(1,2)$, $(3,4)$ and $(1,3)(2,4)$. We are interested in the local system of coefficients of $X'_{9\num d}$ induced by the sign representation on the $5$ singular points. Since the involution $(1,2)$ interchanges exactly $2$ singular points, the twisted Borel--Moore homology of $X'_{9\num d}$ is contained in the part of the Borel--Moore homology of $Y_{9\num d}$ which is alternating under $(1,2)$. Therefore, as the whole Borel--Moore homology of $Y_{9\num d}$ is invariant, the twisted Borel--Moore homology of $X'_{9\num d}$ must vanish. Furthermore, the structure of $\Phi'_{9\num d}$ and $F'_{9\num d}$ as fibrations over $X'_{9\num d}$ also implies that the Borel--Moore homology of these spaces vanishes.
\end{proof}

\begin{proof}[Proof of Lemma~\ref{hom9e}]
Let us consider the configuration space
$$
Y_{9\num e}:=\left\{(\{\alpha,\beta\},a_1,a_2,a_3,a_4)\in\B2{\Pp2}\times\tF4{\Pp2}\left|
\begin{array}{c}\dim\la \alpha,\beta,a_1,a_2\ra = 1,\\ \{\alpha,\beta\}\neq\{a_1,a_2\}\end{array}
\right.
\right\}.
$$

Let us choose a standard frame $(e_1,e_2,e_3,e_4)\in\tF4{\Pp2}$ and identify $\Pp1$ with the line $e_1e_2$. Then $Y_{9\num e}$ is isomorphic to the product $(\B2{\Pp1}\setminus\{\text{pt}\})\times \PGL(3)$ by the map sending $(\{\alpha,\beta\},a_1,a_2,a_3,a_4)$ to $(\{\phi(\alpha),\phi(\beta)\},\phi)$ where $\phi$ is the unique automorphism of $\Pp2$ such that $\phi(a_i)=e_i$ for all $i=1,\dots,4$. In particular, this implies that the  Borel--Moore homology of $Y_{9\num e}$ with constant coefficients has HG polynomial $t^4\Ll^{-2}-t$.

The map $Y_{9\num e}\rightarrow X'_{9\num e}$ given by $(\{\alpha,\beta\},a_1,a_2,a_3,a_4)\mapsto(\{\alpha,\beta\},\{a_1,a_2,a_3,a_4,a_5\})$ with $a_5$ the intersection point of the lines $a_1a_3$ and $a_2a_4$ allows to identify $X'_{9\num e}$ with the quotient of $Y_{9\num e}$ by the involution 
$$i\co (\{\alpha,\beta\},a_1,a_2,a_3,a_4)\longmapsto(\{\alpha,\beta\},a_2,a_1,a_4,a_3).$$

Since $i$ interchanges two pairs of singular points, the twisted Borel--Moore homology of $X'_{9\num e}$ coincides with the part of the Borel--Moore homology which is invariant under $i$. Then the claim follows from the fact that the Borel--Moore homology of both factors $\PGL(3)$ and $\B2{\Pp1}\setminus\{\text{pt}\}$  of $Y_{9\num e}$ is invariant under the $\s_2$-action induced by  $i$.
\end{proof}

\section{Configurations of type $10$}\label{sixpts}

In this section, we deal with the singular sets of singular quartics with a marked bitangent that are the union of four lines in general position. Such a quartic has $6$ distinct singular points, hence  a singular set of the configuration space $X'_{10}$ will be an element of $\B2{\Pp2}\times\B6{\Pp2}$. As always, we will denote by $\pm\Q$ the pull-back of the local system $\pm\Q$ under the forgetful map $X'_{10}\rightarrow\B6{\Pp2}$.

\begin{lem}\label{hom10}
\begin{equation}\label{eq10}
\begin{array}{l}
\wp(\BM[\pm\Q]\pu {X'_{10}})=(\Ll^{-2}t^4+1)\cdot\wp(\BM\pu{\PGL(3)}),\\
\wp(\BM \pu{\Phi'_{10}})=t^5(\Ll^{-2}t^4+1)\cdot\wp(\BM\pu{\PGL(3)}),\\
\wp(\BM \pu{F'_{10}})=t^7\Ll^{-1}(\Ll^{-2}t^4+1)\cdot\wp(\BM\pu{\PGL(3)}).
\end{array}
\end{equation}
\end{lem}

\proof
We start by observing that it suffices to prove the description of the Borel--Moore homology of the configuration space $X'_{10}$. The results on $\Phi'_{10}$ and $F'_{10}$ will immediately follows from their structures as simplicial bundle, resp., vector bundle.

Recall that the elements of $X'_{10}$ are configurations $(\{\alpha,\beta\},K)$ such that $K$ is the singular set of $\mathcal C=\bigcup_i\ell_i\subset\Pp2$ for a configuration of $4$ lines $\ell_1,\ell_2,\ell_3,\ell_4$ in general position  and the line $\tau=\alpha\beta$ is either tangent to $\mathcal C$ at the points $\alpha$ and $\beta$, or it is contained in $\mathcal C$. This implies that we may view $X'_{10}$ as a subset of $\B2{\Pp2}\times\tB4{\Pp2\duale}$, and that $X'_{10}$ has two irreducible components:
$$X'_{10\num a}=\{(\{\alpha,\beta\},\{\ell_1,\dots,\ell_4\})\in X'_{10}| \alpha\beta\subset\bigcup_i\ell_i\}$$
and
$$
X'_{10\num b}=\{(\{\alpha,\beta\},\{\ell_1,\dots,\ell_4\})\in X'_{10}| \ell_1\cap \ell_2=\{\alpha\}, \ell_3\cap\ell_4=\{\beta\}\},$$
because the only bitangent lines to the singular quartic $\mathcal C$ are the components of $\mathcal C$ and the lines joining the intersection points of two disjoint pairs of components of $\mathcal C$.

We need to compute the Borel--Moore homology of $X'_{10}$ in the twisted local system of coefficients $\pm\Q$. This local system of coefficients coincides with the restriction to $X'_{10}$ of the trivial local system on $\B2{\Pp2}\times\tB4{\Pp2\duale}$ under the inclusion $X'_{10}\hookrightarrow\B2{\Pp2}\times\tB4{\Pp2\duale}$. This follows from the fact that interchanging two lines $\ell_i,\ell_j$ interchanges two pairs of singular points of $\mathcal C$, thus inducing a permutation of even sign in the configuration of six singular points.

We proceed to consider configurations of type $10\num a$. Notice that without loss of generality we can always assume that the marked points $\alpha,\beta$ lie on the line $\ell_4$. In other words, we can obtain $X'_{10\num a}$ as the quotient of the space 
$$
Y_{10\num a}=\{(\{\alpha,\beta\},(\ell_1,\dots,\ell_4))\in\B2{\Pp2}\times\tB4{\Pp2\duale}| \alpha,\beta\in\ell_4\}
$$
by the action of $\s_3$ permuting the lines $\ell_1,\ell_2,\ell_3$.

The space $Y_{10\num a}$ is fibred over $\tF4{\Pp2}\cong\PGL(3)$ with fibre isomorphic to $\B2{\ell_4}\cong\B2{\Pp1}$. The space $\B2{\Pp1}$ is isomorphic to $\Sym^2\Pp1$ with the diagonal removed, i.e. to the complement of a smooth conic in $\Pp2$. Hence we have 
$$\wp(\BM\pu{Y_{10\num a}})=t^4\Ll^{-2}\wp(\BM\pu{\PGL(3)},$$
and since the whole Borel--Moore homology of $Y_{10\num a}$ is invariant under the $\s_3$-action, this yields the Borel--Moore homology of $X'_{10\num a}$ as well.

Analogously, we realize $X'_{10\num b}$ as the quotient of $\tB4{\Pp2\duale}$ the action of the group generated by the involution interchanging $\ell_1\leftrightarrow\ell_2$ and the involution $\ell_1\leftrightarrow\ell_3$, $\ell_2\leftrightarrow\ell_4$. Since the Borel--Moore homology of $\tB4{\Pp2\duale}\cong\PGL(3)$ is invariant under any permutation of the points, the Borel--Moore homology of $X'_{10\num b}$ coincides with that of $\PGL(3)$. Then the claim follows from the fact that the Borel--Moore homology of $X'_{10}$ is the direct sum of the Borel--Moore homology of its two components $X'_{10\num a}$ and $X'_{10\num b}$.
\qed

\section{Quartic curves with a flex bitangent}\label{Qdelta}

In this section, we will compute the rational cohomology of the moduli space $\qmin_\delta $ of pairs $(C,\tau)$ such that $C$ is a smooth quartic curve and $\tau$ a flex bitangent.

\begin{thm}\label{cohqdelta}
The rational cohomology of $\qmin_\delta$ is one-dimensional and concentrated in degree $0$.
\end{thm}

Recall that $\imin_\delta$ is fibred over the space $\Pp{}(T_{\Pp2})$, which can be viewed as the incidence correspondence
$$\{(\alpha,\tau)\in\Pp2\times\Pp2\duale| \alpha\in\tau\}.$$

We start by considering the fibre of the map $\pmin_\delta \co \imin_\delta \rightarrow\Pp{}(T_{\Pp2})$ over $(p,t)\in\Pp{}(T_{\Pp2})$. 
Set $t'=t\setminus\{p\}.$ 
Consider the $11$-dimensional complex vector space 
$$V_{p,t}:=\left\{
f\in V\left| \begin{array}{c} 
\text{ the line t is either contained in $\van f$ or it}\\
\text{or is a flex bitangent to $t$ at the point $p$} \end{array}\right.\right\}.$$

Then the fibre $(\pmin_\delta )^{-1}(p,t)$ is equal to $V_{p,t}\setminus\Sigma$. Hence, the fibre of $\pmin_\delta $ can be viewed as the complement of a discriminant in the vector space $V_{p,t}$. In particular, its cohomology can be computed using Vassiliev--Gorinov's method.

We proceed by giving the classification of the singular sets in $\Pp2$ of quartic curves that pass through $p$ and have the line $t\ni p$ as flex bitangent. These are exactly the singular sets of the elements of $V_{p,t}\cap\Sigma$. 
First, we classify in Table~\ref{flex1-6} the singular sets that come from refining singular configurations of type~$1$--$6$ in Vassiliev's list (Table~\ref{Vaslist}).
\begin{table}
\caption{\label{flex1-6}Singular configurations of type $1$--$6$ and associated strata.}
{\small
\begin{tabular}{l@{\ }p{11cm}}
1a&The point $p$.\\      
&Stratum: $F_{1\num a}$ is isomorphic to $\C^{10}$.\\
1b&One point on $t'$.\\
&Stratum: $F_{1\num b}$ is a $\C^9$-bundle over $t'\cong\C$.\\
1c&A points outside $t$.\\
&Stratum: $F_{1\num c}$ is a $\C^8$-bundle over the affine space $\C^2$.\\
2a&Two points on $t$.\\
&Stratum: $F_{2\num a}$ is a $\C^8$-bundle over $\Phi_{2\num a}$, which is a non-orientable $\op\Delta_1$-bundle over $X_{2\num a}\cong\B2{\Pp1}$.\\
2b&The point $p$ and a point outside $t$.\\
&Stratum: $F_{2\num b}$ is a $\C^7$-bundle over $\Phi_{2\num b}$, which is a non-orientable $\op\Delta_1$-bundle over $X_{2\num b}\cong\C^2$.\\
2c&A point on $t'$ and a point outside $t$.\\
&Stratum: $F_{2\num c}$ is a $\C^6$-bundle over $\Phi_{2\num c}$, which is a non-orientable $\op\Delta_1$-bundle over $X_{2\num c}\cong\C^3$.\\
2d&Two points outside $t$.\\
&The Borel--Moore homology of $F_{2\num d}$ and $\Phi_{2\num d}$ vanishes, because they are fibred over the configurations space $\B2{\C^2}$, whose twisted Borel--Moore homology vanishes.\\
3x&Configurations with three collinear points.\\
&To get all strata, we have to distinguish whether the line is $t$, 
if one of the singular points coincide with $p$ 
or lies on $t$
or if the configuration is general. 
In each of these cases the space $X_{\num{3x}}$ admits a locally fibration with fibre isomorphic to either $\B3{\Pp1}$ or $\B2{\C^2}$. Since the twisted Borel--Moore homology of both these configuration spaces vanishes, configurations of type $3$ do not contribute to the Borel--Moore homology of $\mathcal D^-_\delta$.\\
4a&Two points on $t$ and one point outside $t$.\\
&The stratum  $F_{4\num a}$ is a $\C^5$-bundle over $\Phi_{4\num a}$, which is a $\op\Delta_2$-bundle over the configuration space $X_{4\num a}=\B2{\Pp1}\times\C^2$.\\
4x&All other configurations of three points in linear general position.\\
&We have to consider the following subcases: one of the points is $p$ and the other two lie outside $t$ ($\C^4$-bundle), a singular point lies on $t'$ and the other two lie outside $t$ ($\C^3$-bundle), all three points lie outside $t$ (case $4^*$, $\C^2$-bundle). All these subcases correspond to configuration spaces with trivial twisted Borel--Moore homology.\\
5x&A line in $\Pp2$.\\
&We have to distinguish whether the line equals $t$, passes through $p$ of not. In each case $\num x$ the Borel--Moore homology of $\Phi_{\num{5x}}$ is trivial, because the singular configuration contains a rational curve.\\
6x&Three collinear points $p_1,p_2,p_3$ and a fourth point $q$ in general linear position.\\ 
& We have to distinguish between the following subcases: the three collinear points belong to $t$ ($\C^4$-bundle), a point $p_i$ and $q$ both lies on $t$ ($\C^3$-bundle), $p_1$ belongs to $t$ but and $p_2, p_3$ and $q$ do not ($\C^3$-bundle), none of the $p_i$ lies on $t$ but $q\in t$ ($\C^2$-bundle), $p_1\in t$ but all other point lie outside $t$ (case $6^*$, $\C$-bundle). In each of these cases the configuration space is fibred over a base space with fibre isomorphic to either $\B2{\C^2}$, $\B2{\C}$ or $\B3{\Pp1}$. For this reason, all configurations of type $6$ contribute trivially to the Borel--Moore homology of $\mathcal D^-_\delta$.\\ 
\end{tabular}
}
\end{table}

For configurations $K$ of type $7$--$13$ in Table~\ref{Vaslist} we can distinguish whether the general curve singular at $K$ will not contain the line $t$ 
(in which case we will call it a {configuration of the first kind}) or if every curve singular at $K$ will contain $t$ ({configuration of the second kind}). 

It is easy to see that if the $\van f$ does not contain $t$, and its singular locus contains a configuration $K$ of type $7$--$12$ in Vassiliev's classification, then the singular locus $K_f\subset\Pp2$ is either a line through $p$, or a conic (possibly singular) tangent to $t$ at the point $p$. Hence, a configuration $K$ of the first kind will either contain a rational curve, or a finite number of point lying on a rational curve. In this way we can prove that configurations of the first kind do not contribute to the spectral sequence converging to the Borel--Moore homology of $V_{p,t}\cap \Sigma$. Therefore, it suffices for us to consider the strata associated to configurations of the second kind, which we list in Table~\ref{flex-seckind}.

\begin{table}
\caption{\label{flex-seckind}Singular configurations of the second kind and associated strata.}
{\small
\begin{tabular}{l@{\ }p{11cm}}
$7'$&Two points on $t$ and two other points (no three points in the configurations are allowed to be collinear).\\
&Stratum: The general curve singular in such a configuration $K$ is the union of the line $t$, the line passing through the two points in $K\setminus t$ and a conic passing through the points of $K$. The configuration space $X_{7'}$ is contained in $\B2{t}\times\B2{\C^2}$ with complement a fibration over $\F2{\Pp1}$ with fibre isomorphic to $\B2\C$. Since the twisted Borel--Moore homology of both $\B2{\C^2}$ and $\B2\C$ vanish by Lemma~\ref{lem1}, the twisted Borel--Moore homology of $X_{7'}$ and the Borel--Moore homology of $\Phi_{7'}$ and $F_{7'}$ are trivial as well.\\
8a&The union of the line $t$ and a point outside $t$.\\
&The strata $F_{8\num a}$ and $\Phi_{8\num a}$ have trivial Borel--Moore homology, because the singular configurations of this type contain a rational curve.\\
8b& The union of a line different from $t$ and a point on $t$.\\
&The Borel--Moore homology of the strata $\Phi_{8\num b}$ and $F_{8\num b}$ is trivial, because configurations of type $\num{8\num b}$ always contain a rational curve.\\
9a&Five points $\{a,b,c,d,e\}$ with $a,b\in t$, $c,d\in(\Pp2\setminus t)$ and $\{e\}=t\cap cd$.\\
&Stratum: The configuration space $X_{9\num a}$ is isomorphic to $X_{7'}$, hence its twisted Borel--Moore homology vanishes. Therefore also the Borel--Moore homology of $\Phi_{9\num a}$ and of $F_{9\num a}$ is trivial.\\
9b&Five points $\{a,b,c,d,e\}$ with $a,b\in t$, $c,d\in(\Pp2\setminus t)$ and $\{e\}=t\cap cd$.\\
&Stratum: The stratum $F_{9\num b}$ is a $\C$-bundle over $\Phi_{9\num b}$, which in turn is a $\op\Delta_4$-bundle over the configuration space $X_{9\num b}$, which is the quotient of the space $\{(a,b,c,d)\in\tF4{\Pp2}| a,b\in t\}$ by the equivalence relation generated by $(a,b,c,d)\sim(b,a,d,c)$. \\
$10'$&Six points which are the pairwise intersection of four lines in general position, one of which is $t$.\\
&Stratum: The stratum $F_{10'}$ is a $\C$-bundle over $\Phi_{10'}$, which in turn is a $\op\Delta_5$-bundle over the configuration space $X_{10'}$ which is isomorphic to the configuration space of three unordered lines in general position and such that their union intersects $t$ in three distinct points.\\ 
$12'$&The union of $t$ and another line.\\
&The strata $\Phi_{12'}$ and $F_{12'}$ have trivial Borel--Moore homology. This is a consequence of the fact that the singular locus contains a rational curve.\\
13&The whole projective plane.
\end{tabular}
}
\end{table}

In view of the description of the strata given there, it suffices to deal with the configuration types $9\num b$ and $10'$. 
The Borel--Moore homology of $X_{9\num b}$ and $X_{10'}$ is not difficult to compute. However, we will not need this result, because it is possible to prove that the contributions of these two strata kill each other in the spectral sequences associated to the stratifications $\Phi_\pu$ and $F_\pu$.

\begin{lem}\label{9+10}
The Borel--Moore homology of $\Phi_{9\num b}\cup\Phi_{10'}\subset\ba\Lambda\ba$  is trivial, and the same holds for $F_{9\num b}\cup F_{10'}\subset\ba\mathcal X\ba$.
\end{lem}
\proof
The stratum $\Phi_{10'}$ is a simplicial bundle over $X_{10'}$. Its fibre over a configuration $K\in X_{10'}$ is an open $5$-dimensional simplex whose vertices are in canonical correspondence with the six points in $K$.
We can partially compactify $\Phi'_{10'}$ by considering the simplicial bundle $\Psi$ over $X_{10'}$ whose fibres are closed $5$-dimensional simplices, in such a way that the fibres of $\Phi_{10'}$ coincide with the interiors of the fibres of $\Psi\rightarrow X_{10'}$. 
The simplicial bundle $\Psi$ is contained $\ba\Lambda\ba$, where it can be realized as the union of the simplices corresponding to subsets of the configurations in $X_{10'}$. 

Observe that every configuration $K\in X_{10'}$ contains exactly three points lying outside $t$. They correspond to a $2$-dimensional face $D_K$ of the fibre $\Psi_K$ of $\Psi\rightarrow X_{10'}$ lying over $K$. Let us define $\Chi_K$ to be the union of the interior of all the faces of the $5$-dimensional simplex $\Psi_K$ that contain the interior of  $D_K$.
The complement of $\Chi_K$ in $\Psi_K$ is the union of all closed faces that do not contain the three vertices of $D_K$. The Borel--Moore homology of $\Chi_K$ coincides with the relative homology of the pair $(\Psi_K,\Psi_K\setminus\Chi_K)$, which is trivial because both spaces can be contracted to the same point. 

Let us consider the subset $\Chi\subset\Psi$ given by the union of the $\Chi_K$ for all $K\in X_{10'}$. Then the Borel--Moore homology of $\Chi$ is trivial as well. On the other hand, we can view $\Chi$ as the disjoint union of open simplices of dimension varying from two to five. For $k=2,\dots,5$, denote by $\Chi^{(k)}$ the union of the interior of all $k$-dimensional faces of simplices contained in $\Chi$. 

The space $\Chi^{(2)}$ is fibred over $X_{10'}$ with fibre the interior of $D_K$. It coincides with the stratum $\Phi_{4^*}$ coming from configurations of type $4^*$, containing three points in general linear position not lying on $t$ (see Table~\ref{flex1-6}). In particular, the Borel--Moore homology of $\Chi^{(2)}$ is trivial.
Analogously, the stratum $\Chi^{(3)}$ coincides with the stratum $\Phi_{6^*}$ corresponding to configurations of type $6^*$ (three collinear points of which exactly one lies on $t$ and a fourth point not lying on $t$ and not collinear with the others). As we explained in Table~\ref{flex1-6}, the Borel--Moore homology of $\Chi^{(3)}$ is trivial.

Hence, the Borel--Moore homology of $\Chi$ coincides with the Borel--Moore homology of the union of its strata $\Chi^{(4)}=\Phi_{9\num b}$ and $\Chi^{(5)}=\Phi_{10'}$. This proves that the Borel--Moore homology of $\Phi_{9\num b}\cup\Phi_{10'}$ is trivial.
As to the second part of the claim, it suffices to observe that $F_{9\num b}\cup F_{10'}$ is a complex line bundle over $\Phi_{9\num b}\cup\Phi_{10'}$.
\qed

Furthermore, also the configuration space $X_{13}=\{\Pp2\}$ contributes trivially to the Borel--Moore homology of $V_{p,t}\cap\Sigma$.

\begin{lem}
\begin{enumerate}
\item
The $\ee^1$ terms of the spectral sequence $$\ee^r_{u,v}\Rightarrow\BM{u+v}{\ba\Lambda\ba\setminus \Phi_{13}}$$ associated with the stratification $\Phi_\pu$ are as given in Table~\ref{flex-first}.
\item The Borel--Moore homology of $F_{13}$ is trivial.
\end{enumerate}
\end{lem}

\proof
The terms of this spectral sequence are given by $\ee^1_{u,v}=\BM v{\Phi_u}$ for all configuration types $u$ such that the Borel--Moore homology of $\Phi_u$ is non-trivial. Furthermore, we can omit all configurations with more than $4$ singular points in view of Lemma~\ref{9+10}.
Then the first part of the claim follows from the description of the strata $\Phi_{j\num x}$ given in Table~\ref{flex1-6}. 

We can observe that only configurations of type $1$--$4$ contribute non-trivial  $\ee^1$ terms. Furthermore, the union of the configuration spaces $X_{j\num x}$ with $j=1,\dots,4$ gives all configurations of $\leq 4$ points in $\Pp2$.
Hence, the reasoning in the proof of Lemma~\ref{first-Lambda} applies also in this case, thus yielding
$$\ee^{\infty}_{1,-1}=0, \ 
\ee^{\infty}_{u,v}=0\text{ for }(u,v)\neq(1,-1).
$$

As a consequence, the Borel--Moore homology of $\ba\Lambda\ba\setminus\Phi_{13}$ is $1$-dimensional and concentrated in degree $0$.
Then the second part of the claim follows from the fact that the $F_{13}$ is an open cone over $\ba\Lambda\ba\setminus\Phi_{13}$ in view of Proposition~\ref{ucci}.
\qed

\begin{table}
  \caption{\label{flex-first} $\ee^1$ terms of the spectral sequence $\ee^r_{u,v}\Rightarrow\BM{u+v}{\ba\Lambda\ba\setminus\Phi_{13}}$.}
$$\begin{array}{r|cccccccc}
v&&&&&&&\\[6pt]
 1&0&0&\Q(2)&0&0&\Q(3)&\Q(3)&\\
 0&0&\Q(1)&0&0&\Q(2)&0&0&\\
-1&\Q&0&0&\Q(1)&0&0&0&\\
\hline
&1&2&3&4&5&6&7&u\\
\num{type}&\num{(1a)}&\num{(1b)}&\num{(1c)}&\num{(2a)}&\num{(2b)}&\num{(2c)}&\num{(4a)}&
\end{array}
$$
\end{table}

We are ready to calculate the Borel--Moore homology of $V_{p,t}\cap\Sigma$.
\begin{lem}\label{cohXpt}
\begin{enumerate}
\item The $\Ee^1$ terms of the spectral sequence $$\Ee^r_{u,v}\Rightarrow\BM{u+v}{V_{p,t}\cap\Sigma}$$ are as given in Table~\ref{flex-second}. This spectral sequence degenerates at $\Ee^1$.
\item The rational cohomology of $V_{p,t}\setminus\Sigma$ has HG polynomial $(1-t\Ll)^3$.
\end{enumerate}
\end{lem}

\begin{table}
  \caption{\label{flex-second} $\Ee^1$ terms of the spectral sequence $\Ee^r_{u,v}\Rightarrow \BM{u+v}{\ba\mathcal X\ba}=\BM{u+v}{V_{p,t}\cap\Sigma}$.}
$$\begin{array}{r|cccc}
v&&&&\\[6pt]
19&\Q(10)^{\oplus3}&0&0&\\
18&0&0&0&\\
17&0&\Q(9)^{\oplus3}&0&\\
16&0&0&0&\\
15&0&0&\Q(8)&\\
\hline
&1&2&3&u\\
\num{type}&(1\num x)&(2\num x)&(4\num x)&
\end{array}
$$

\end{table}
\proof 
We have $\Ee^1_{u,v}=\BM{u+v}{F_{u}}$, where $u$ refers to the $u$th configuration type in our list. Since $F_u$ is a vector bundle of a certain rank $k_u$ over $\Phi_{u}$, one has $\Ee^1_{u,v}=\ee^1_{u,v-2k_u}\otimes\Q(k_u)$. This allows to compute the $\Ee^1_{u,v}$ as in Table~\ref{flex-second}. Degeneracy at $\Ee^1$ follows immediately from the shape of the spectral sequence. The result on the cohomology of the complement of the discriminant $V_{p,t}\setminus\Sigma$ follows from Alexander's duality~\eqref{alexander}.
\qed

\begin{table}
\caption{\label{flex-leray}
Leray spectral sequence of the fibration $\imin_\delta\rightarrow \Pp{}(T_{\Pp2})$}
$$\begin{array}{r|cccccccc}
q&&&&&&&&E^2_{p,q}\\[6pt]
3&\Q(-3)  &0&\Q(-4)^2&0&\Q(-5)^2&0&\Q(-6)\\
2&\Q(-2)^3&0&\Q(-3)^6&0&\Q(-4)^6&0&\Q(-5)^3\\
1&\Q(-1)^3&0&\Q(-2)^6&0&\Q(-3)^6&0&\Q(-4)^3\\
0&\Q&0&\Q(-1)^3&0&\Q(-2)^3&0&\Q(-3)
\\\hline
&0&1&2&3&4&5&6&p\\
\end{array}
$$

$$\begin{array}{r|cccccccc}
q&&&&&&&&E^3_{p,q}\\[6pt]
3&  0     &0&0       &0&0       &0&\Q(-6)\\
2&0       &0&\Q(-3)  &0&\Q(-4)  &0&\Q(-5)  \\
1&\Q(-1)  &0&\Q(-2)  &0&\Q(-3)  &0&0       \\
0&\Q&0&0     &0&0     &0&0
\\\hline
&0&1&2&3&4&5&6&p\\
\end{array}
$$
\end{table}

\begin{proof}[Proof of Theorem~\ref{cohqdelta}]
We want to compute the cohomology of $\imin_\delta$ by using the Leray spectral sequence associated to the fibration $\imin_\delta\rightarrow\Pp{}(T_{\Pp2})$ with fibre $V_{p,t}\setminus \Sigma$. From the fact that $\Pp{}(T_{\Pp2})$ is a $\Pp1$-bundle over $\Pp2$ and from the computation of the cohomology of the fibre given in Lemma~\ref{cohXpt} above we get that the $E_2$ term of the Leray spectral sequence is as given in the first part of Table~\ref{flex-leray}.

To compute the differentials of the spectral sequence, we keep in mind that we proved in Lemma~\ref{division} that it has to be a tensor product of the cohomology of $\GL(3)$. There is only one possible behaviour of the differentials that would ensure such a divisibility: This is the case in which the $E^3$ term is as in the second part of Table~\ref{flex-leray} and the spectral sequence degenerates at $E^3$. This yields
$$\coh\pu{\imin_\delta}\cong\coh\pu{\GL(3)}.$$
This isomorphism implies the claim, since by Lemma~\ref{division} we also have $\coh\pu{\imin_\delta}\cong\coh\pu{\qmin_\delta}\otimes \coh\pu{\GL(3)}$.
\end{proof}

\section{Vassiliev--Gorinov's method}\label{VGmethod}
In order to make the article as self-contained as possible, we include here an introduction to Vassiliev--Gorinov's method for computing the cohomology of complements of discriminants, following \cite{OTM4} and \cite{OTtesi}. This review of the method is by no means complete, and we encourage the interested reader to consult \cite{Vart}, \cite{Gorinov} and \cite{OTM4}.

Let $Z$ be a projective variety, $\mathcal F$ a vector bundle on $Z$ and $V$ the space of global sections of $\mathcal F$. Define the discriminant $\Sigma\subset V$ as the locus of sections with a vanishing locus which is either singular or not of the expected dimension. Assume that $\Sigma$ is a subvariety of $V$ of pure codimension $1$. 
Our aim is to  compute the rational cohomology of the complement of the discriminant, $X=V\setminus \Sigma$. This is equivalent to determining the Borel--Moore homology of the discriminant, because there is an isomorphism between the reduced cohomology of $X$ and Borel--Moore homology of $\Sigma$. If we denote by $M$ the dimension of $V$, this isomorphism can be formulated as
\begin{equation}\label{alexander}\tilde H^{\pu}(X;\Q)\cong \bar H_{2M-\pu-1}(\Sigma;\Q)(-M).\end{equation}

\begin{dfn}
A subset $S\subset Z$ is called a \emph{configuration} in $Z$ if it is compact and non-empty. The space of all configurations in $Z$ is denoted by $\Conf(Z)$.
\end{dfn}

\begin{prop}[\cite{Gorinov}]
The Fubini--Study metric on $Z$ induces in a natural way on $\Conf(Z)$ the structure of a compact complete metric space.
\end{prop}

To every element in $v\in V$, we can associate its singular locus $K_v\in\Conf(Z)\cup\{\emptyset\}$. We have that $K_0 $ equals $Z$, and that $L(K):=\{v\in V: K\subset K_v\}$ is a linear space for all $K\in \Conf(Z)$. 

Vassiliev--Gorinov's method is based on the choice of a collection of families 
of configurations $X_1,\ldots,X_R\subset\Conf(Z)$, satisfying some axioms (\cite[3.2]{Gorinov}, \cite[List 2.1]{OTM4}). Intuitively, we have to start by classifying all possible singular loci of elements of $V$. Note that singular loci of the same type have a space $L(K)$ of the same dimension. We can put all singular configurations of the same type in a family. Then we order all families we get according to the inclusion of configurations. In this way we obtain a collection of families of configurations which may already satisfy Gorinov's axioms. If this is not the case, the problem can be solved by adding new families to the collection. Typically, the elements of these new families will be degenerations of configurations already considered. 
For instance, configurations with three points on the same projective line and a point outside it can degenerate into configurations with four points on the same line, even if there is no $v\in V$ which is only singular at four collinear points.  

Once the existence of a collection $X_1,\dots,X_R$ satisfying Gorinov's axioms is established, Vassi\-liev--Gorinov's method gives a recipe for constructing a space $\ba \mathcal X\ba$ and a map
$$|\epsilon|\co  \ba\X\ba\longrightarrow \Sigma,$$
called \emph{geometric realization}, which is a homotopy equivalence and induces an isomorphism on Borel--Moore homology.
The original construction by Vassiliev and Gorinov uses topological joins to construct $\ba \X\ba$. This construction was reformulated in \cite{OTM4} by using the language of cubical spaces. This ensures in particular that the map induced by $|\epsilon|$ on Borel--Moore homology respects mixed Hodge structures.

Vassiliev--Gorinov's method provides also a stratification $\{F_j\}_{j=1,\dots,N}$ on $\ba\X\ba$. Each $F_j$ is locally closed in $\ba\X\ba$, hence one gets a spectral sequence converging to $\bar H_\pu(\Sigma;\Q)\cong\bar H_\pu(\ba\X\ba;\Q)$, with $\Ee^1_{p,q}\cong\bar H_{p+q}(F_p)$. To compute the Borel--Moore homology of $F_j$ for all $j=1,\dots,R$, it is helpful to use an auxiliary space $\ba\Lambda\ba$, whose construction depends only on the geometry of the families $X_1\dots,X_R$, and which is covered by locally closed subsets $\{\Phi_j\}_{j=1,\dots,N}$. 

\begin{prop}[\cite{Gorinov}]\label{ucci}
\renewcommand{\labelenumi}{\arabic{enumi}.}
\begin{enumerate}
\item\label{ucpri} For every $j=1,\dots,R$, the stratum $F_j$ is a complex vector bundle  over $\Phi_j$. The space $\Phi_j$ is in turn a fiber bundle over the configuration space $X_j$. 
\item\label{ucsec} If $X_j$ consists of configurations of $m$ points, the fiber of $\Phi_j$ over any $x\in X_j$ is an $(m-1)$-dimensional open simplex, which changes its orientation under the homotopy class of a loop in $X_j$ interchanging a pair of points in $x_j$.
\item\label{opeco}  If $X_R=\{Z\}$, $F_R$ is the open cone with vertex a point (corresponding to the configuration $Z$),  over $\ba\Lambda\ba\setminus\Phi_R$. 
\end{enumerate}\end{prop}

We recall here the topological definition of an open cone.
\begin{dfn}
Let $B$ be a topological space. Then a space is said to be an \emph{open cone} over $B$ with vertex a point if it is homeomorphic to the space $B\times[0,1)/R$, where the equivalence relation is $R=(B\times\{0\})^2$.
\end{dfn}

The fiber bundle $\Phi_j\rightarrow X_j$ of Proposition \ref{ucci} is in general non-orientable. As a consequence, we have to consider the homology of $X_j$ with coefficients not in $\Q$, but in some local system of rank one. Therefore we recall some constructions concerning Borel--Moore homology of configuration spaces with twisted coefficients.

\begin{dfn}\label{tuples}
Let $Z$ be a topological space. Then for every $k\geq 1$ we have the space of ordered configurations of $k$ points in $Z$,
$$\F kZ=Z^k\setminus\bigcup_{1\leq i<j\leq k}\{(z_1,\dots,z_k)\in Z^k: z_i=z_j\}.$$
There is a natural action of the symmetric group $\s_k$ on $F(k,Z)$. The quotient is called the space of unordered configurations of $k$ points in $Z$,
$$\B kZ=\F kZ/\s_k.$$
\end{dfn}

The \emph{sign representation} $\pi_1(\B kZ)\rightarrow \Aut(\Z)$ maps the paths in $\B kZ$ defining odd (respectively, even) permutations of $k$ points to multiplication by $-1$ (respectively, 1). The local system $\pm\Q$ over $\B kZ$ is the one locally isomorphic to $\Q$, but with monodromy representation equal to the sign representation of $\pi_1(\B kZ)$. We will often call $\bar H_\pu (\B kZ,\pm\Q)$ the \emph{Borel--Moore homology of $\B kZ$ with twisted coefficients}, or, simply, the \emph{twisted Borel--Moore homology of $\B kZ$}.

The following is Lemma~2 in \cite{Vart}.
\begin{lem}\begin{enumerate}\label{lem1}
\item\label{lem1i}
If $N\geq 1$, $k\geq2$, the twisted Borel--Moore homology of $\B k{\C^N}$ is trivial.
\item\label{lem1ii}
If $N\geq1$, we have isomorphisms $$\bar H_\pu(\B k{\Pp N};\pm\Q)\cong\bar H_{\pu-k(k-1)}(\bb G(k-1,\Pp N);\Q)$$ for every $k\geq1$, where $\bb G(k-1,\Pp N)$ denotes the Grassmann variety of $(k-1)$-dimensional linear subspaces in $\Pp N$. In particular, $\bar H_\pu(\B k{\Pp N};\pm\Q)=0$ if $k\geq N+2$.
\end{enumerate}
\end{lem}

\def\cprime{$'$}

\end{document}